\documentclass[11pt]{amsart} 
\usepackage[english]{babel}
\usepackage{mathtools, float}
\usepackage{amssymb,amscd, graphics,color,latexsym,cancel, booktabs}   
\usepackage[linktoc=all, pagebackref, hyperindex]{hyperref}%
\hypersetup{colorlinks,  citecolor=blue,  filecolor=blue,  linkcolor=blue,  urlcolor=black}
\usepackage{multirow, graphicx}
\usepackage{comment}

\usepackage{tikz-cd} 
\usepackage{extpfeil}
\usepackage{tikz}
\usetikzlibrary{matrix,calc}
\usepackage{mathrsfs}
\usepackage[all]{xy}
\usepackage{amsfonts}
\usepackage[normalem]{ulem}
\usepackage{euscript}
\usepackage{amsmath}
\usepackage{xcolor}
\usepackage{ifpdf}
\usepackage{cleveref}
\LARGE\textwidth=6in
\newcommand\restr[2]{{
  \left.\kern-\nulldelimiterspace 
  #1 
  \vphantom{\big|} 
  \right|_{#2} 
}}

\newcommand{\lar}{\longrightarrow}

\newtheorem{Theorem}{Theorem}[section]
\newtheorem{Lemma}[Theorem]{Lemma}

\newtheorem{Proposition}[Theorem]{Proposition}

\theoremstyle{definition}
\newtheorem{Remark}[Theorem]{Remark}
\newtheorem{Example}[Theorem]{Example}

\newtheorem{Definition}[Theorem]{Definition}
\newtheorem{Question}[Theorem]{Question}
\def\sqr#1#2{{\vcenter{\hrule height.#2pt
			\hbox{\vrule width.#2pt height#1pt \kern#1pt
				\vrule width.#2pt}
			\hrule height.#2pt}}}
\def\phi{\varphi}

\def\VaVa{{\mathcal V}\kern-5pt {\mathcal V}}
\def\gr#1#2{{\rm gr}\, _{#1}(#2)}
\def\gr{{\rm gr}\,}
\def\hht{{\rm ht}\,}
\def\depth{{\rm depth}\,}

\def\Min{{\rm Min}\,}
\def\codim{{\rm codim}\,}

\def\ker{{\rm ker}\,}
\def\grade{{\rm grade}\,}
\def\rk{\rm rank}

\def\syz{\mbox{\rm Syz}}
\def\cok{\mbox{\rm coker}}

\def\Ext#1#2#3#4{{\rm Ext}\,^{#1}_{#2}({#3},{#4})}

\def\supp#1{{\rm Supp}\, (#1)}

\def\ini{\mbox{\rm in}}

\def\der#1{\mbox{\rm der}_k(#1)}

\def\cl#1{{\mathcal #1}}

\def\phi{\varphi}

\def\hht{{\rm ht}\,}
\def\grade{{\rm grade}\,}

\def\fm{{\mathfrak m}}

\def\NN{\mathbb N}

\def\fp{{\mathfrak p}}

\def\fm{{\mathfrak m}}

\def\NN{\mathbb N}

\def\cl#1{{\cal #1}}
\def\rk{\rm rank}

\newcommand{\excise}[1]{}

\def\NZQ{\mathbb}               
\def\NN{{\NZQ N}}

\def\PP{{\NZQ P}}

\def\G{{\mathcal G}}

\def\R{{\mathcal R}}

\def\opn#1#2{\def#1{\operatorname{#2}}} 
\opn\chara{char} \opn\length{\lambda} \opn\pd{pd} \opn\rk{rk}
\opn\projdim{proj\,dim} \opn\injdim{inj\,dim} \opn\rank{rank}
\opn\depth{depth} \opn\grade{grade} \opn\height{height}
\opn\embdim{emb\,dim} \opn\codim{codim}

\opn\Tr{Tr} \opn\bigrank{big\,rank}
\opn\superheight{superheight}\opn\lcm{lcm}
\opn\trdeg{tr\,deg}
	\opn\reg{reg} \opn\lreg{lreg} \opn\ini{in} \opn\lpd{lpd}
	\opn\size{size} \opn\sdepth{sdepth}
	\opn\link{link}\opn\fdepth{fdepth}\opn\lex{lex}
	\opn\tr{tr}
	\opn\type{type}
    \opn\Num{Num}   \opn\Str{Str}
	\opn\div{div} \opn\Div{Div} \opn\cl{cl} \opn\Cl{Cl}
	\opn\Spec{Spec} \opn\Supp{Supp} \opn\supp{supp} \opn\Sing{Sing}
	\opn\Ass{Ass} \opn\Min{Min}\opn\Mon{Mon}
	\opn\Ho{H}  \opn\indeg{indeg} \opn\Hs{HS} \opn\Eig{E}
	\opn\Ann{Ann} \opn\Rad{Rad} \opn\Soc{Soc}  \opn\der{Der}
	\opn\Bour{Bour} \opn\proj{Proj}
    \opn\Sym{Sym}
    \opn\Lie{Lie}
	\opn\Im{Im} \opn\Ker{Ker} \opn\Coker{Coker} \opn\Am{Am}
	\opn\Hom{Hom} \opn\Tor{Tor} \opn\Ext{Ext} \opn\End{End}
	\opn\Aut{Aut} \opn\id{id}
	
	\opn\nat{nat}
	\opn\pff{pf}
	\opn\Pf{Pf} \opn\GL{GL} \opn\SL{SL} \opn\mod{mod} \opn\ord{ord}
	\opn\Gin{Gin} \opn\Hilb{HP}\opn\sort{sort}
	\opn\PF{PF}\opn\Ap{Ap}\opn\sat{sat}
    \opn\PGL{PGL}
    \opn\Coh{Coh}
    \opn\op{op}
	\opn\aff{aff} \opn
	\con{conv} \opn\relint{relint} \opn\st{st}
	\opn\lk{lk} \opn\cn{cn} \opn\core{core} \opn\vol{vol}  \opn\inp{inp} \opn\nilpot{nilpot}
	\opn\link{link} \opn\star{star}\opn\lex{lex}\opn\set{set}
	\opn\width{wd}
	\opn\Fr{F}
	\opn\QF{QF}
	\opn\G{G}
    \opn\Id{Id}
	\opn\type{type}\opn\res{res}
	\opn\log{Log}

	\opn\gr{gr}   
	
    \opn\sm{{\rm sm}}
	\def\pot#1#2{#1[\kern-0.28ex[#2]\kern-0.28ex]}

	\opn\dirlim{\underrightarrow{\lim}}
	\opn\inivlim{\underleftarrow{\lim}}
\begin{document}
		
		\title[Numerical invariants for three-generated ideals]{Numerical invariants for three-generated ideals} 
		\author{Marcos Jardim}
		\address{Universidade Estadual de Campinas (UNICAMP) \\ Instituto de Matem\'atica, Estat\'{\i}stica e Computação Cient\'{\i}fica (IMECC) \\ Departamento de Matem\'atica \\ Rua S\'ergio Buarque de Holanda, 651\\ 13083-970 Campinas-SP, Brazil}
		\email{jardim@unicamp.br}
				\author{Felipe  Monteiro}
		\address{Universidade Estadual de Campinas (UNICAMP) \\ Instituto de Matem\'atica, Estat\'{\i}stica e Computação Cient\'{\i}fica (IMECC) \\ Departamento de Matem\'atica \\ Rua S\'ergio Buarque de Holanda, 651\\ 13083-970 Campinas-SP, Brazil}
		\email{ffcmonteiro00@gmail.com}
		\author{Abbas Nasrollah Nejad}
		\address{Department of Mathematics, Institute for Advanced Studies in Basic Sciences (IASBS), Zanjan 45137-66731, Iran \newline \indent
        Instituto de Ciência e Tecnologia, Universidade Federal de S\~ao Paulo (UNIFESP),
        S\~ao Jos\'e dos Campos, SP, Brazil
        }
		\email{annejad@unifesp.br}

        \author{Zaqueu Ramos}
		\address{Universidade Federal de Sergipe\\ Centro de Ci\^encias Exatas e Tecnologias (CCET)\\ Departamento de Matem\'atica\\ Av. Marcelo Deda Chagas, s/n\\ S\~ao Crist\'v\~ao/SE - CEP 49107-230}
		\email{zaqueu@mat.ufs.br}

        \dedicatory{Dedicated to Professor Aron Simis on the occasion of his 84th birthday.}
		
		
		
\begin{abstract}
We study numerical invariants of homogeneous ideals generated by three
forms, with particular emphasis on the interaction between the initial
degree of the syzygy module, the multiplicity of the codimension-two
part of the ideal, and the Bourbaki degree. We extend the Bourbaki degree from gradient ideals of plane curves to arbitrary three-generated homogeneous ideals and obtain sharp numerical bounds in both the general and the equigenerated settings. In the equigenerated case, these bounds extend the du Plessis--Wall picture and lead to the study of numerical, structural, dimensional, and integrable gaps. We show that numerical admissibility does not in general imply realizability and obtain further restrictions from the minimal graded free resolution. As an application, we determine the possible Bourbaki degrees of reduced
singular quartic plane curves and describe them in terms of their singularity configurations; the possible values are $0,1,2,3,4,6,7,8$.
\end{abstract}
	\maketitle

\section*{Introduction}

Three-generated homogeneous ideals form the first genuinely nontrivial
case in codimension two. An ideal of height two generated by two elements
is necessarily a complete intersection. With three generators, the
perfect case is still described by the Hilbert--Burch theorem, but
nonperfect, and hence non-Cohen--Macaulay, behavior already appears.
In both cases the first syzygy module has rank two, and its freeness distinguishes the perfect and nonperfect cases.
There is a natural geometric example of this situation. Let
$X=V(f)\subseteq\PP^2$ be a reduced singular plane curve which is not a cone.
Then its gradient ideal
\[
J_f=(f_x,f_y,f_z)
\]
is minimally generated by three forms of the same degree. The syzygies
of $J_f$ correspond to logarithmic vector fields along $X$, while
$\deg(R/J_f)$ is the global Tjurina number and therefore measures the
singular scheme of the curve. Thus the same rank-two syzygy module
measures, on the algebraic side, the failure of the Hilbert--Burch
situation and, on the geometric side, carries information about the
singularities of the curve. This leads to a basic question: can the
failure of this module to be free be measured by a numerical invariant,
and how does such an invariant interact with $\deg(R/J)$ and, in the
gradient case, with the geometry of the curve?

To make this question precise, let $R=k[x_1,\ldots,x_n]$ be a standard graded polynomial ring over an algebraically closed field of characteristic zero, and let $J=(f_1,f_2,f_3)\subseteq R$ be minimally generated by homogeneous forms
of degrees
\[
1\le d_1\le d_2\le d_3,
\qquad
\gcd(f_1,f_2,f_3)=1.
\]
Set
\[
D=d_1+d_2+d_3,
\qquad
\sigma(\textbf{d}) \doteq \sigma=d_1d_2+d_1d_3+d_2d_3.
\]
If $\nu$ is a minimal homogeneous syzygy of shifted degree $e$, then it
gives an exact sequence
\[
0\lar R(-e)\xrightarrow{\ \nu\ }\syz(J)\lar M_\nu\lar0,
\]
where $M_\nu$ is a torsion-free graded $R$-module of rank one. Theorem~\ref{Bour3Gens}, which is the starting point of the
paper, describes the structure of this module. It is free if and only
if $J$ is perfect of height two. Otherwise,
\[
M_\nu\simeq I_\nu(e-D),
\]
where $I_\nu$ is an unmixed homogeneous ideal of height two, and
\[
\deg(R/I_\nu)+\deg(R/J)=e^2-eD+\sigma.
\]
Here $\deg(R/J)$ means the contribution of the codimension-two
components of $J$. In particular, it is the usual multiplicity when
$\hht(J)=2$ and is zero when $\hht(J)\ge3$.

When $e=\indeg(\syz(J))$, the number $\deg(R/I_\nu)$ does not depend on
the chosen minimal syzygy $\nu$ of initial degree. This gives the
\emph{Bourbaki degree}
\[
\Bour(J)=e^2-eD+\sigma-\deg(R/J),
\]
with the convention that $\Bour(J)=0$ in the perfect height-two case.
Thus the Bourbaki degree relates three pieces of information: the
degrees of the generators, the initial degree of the syzygy module,
and the codimension-two degree of the ideal. Its vanishing characterizes
the Hilbert--Burch case.

The Bourbaki degree was introduced in~\cite{MAA} for gradient ideals of
reduced plane curves. In that setting it measures the failure of the
syzygy module of the gradient ideal to be free. Its vanishing
characterizes the free case, while the Bourbaki formula relates it to
the initial degree of $\syz(J_f)$ and to the global Tjurina number.
Thus the invariant connects the homological structure of the gradient
ideal with the singularities of the curve.

A natural question is how much of this picture depends on the gradient
structure. In the present paper we first remove both the gradient
condition and the assumption that the three generators have the same
degree. We study arbitrary three-generated homogeneous ideals and try
to understand which restrictions come only from three-generation.
We then return to the equigenerated case and finally to gradient ideals
of plane curves. This comparison leads naturally to the realization
and integrability questions considered in the paper.

The general Bourbaki formula already gives strong restrictions when the
generator degrees are arbitrary. Proposition~\ref{prop:basic-numerical-bound}
gives
\[
\Bour(J)\le e^2,
\]
while Proposition~\ref{prop:numerical-bounds-non-equigenerated} gives,
in particular,
\[
d_2<e\le d_1+d_2
\]
and
\[
\Bour(J)\le e^2-eD+\sigma\le d_1d_3.
\]
The last bound is attained exactly in the complete intersection case
with $d_2=d_3$. Proposition~\ref{prop:not-equigen-Bour-consequences}
also characterizes the first values $\Bour(J)=0,1,2$ and describes the
extremal and next-to-extremal cases when the two largest generator
degrees are equal. Moreover,
Example~\ref{ex:next-extremal-non-equigenerated} shows that the value
$d_1d-1$ occurs for every degree type $(d_1,d,d)$ with
$2\le d_1<d$, in fact on a nonempty open family.

These numerical restrictions lead to a realization problem. Fix a
degree vector $\mathbf d=(d_1,d_2,d_3)$. The Bourbaki formula and the preceding bounds give a numerically admissible region for the pairs
\begin{equation}\label{eq:realizable-pairs}
\bigl(\indeg(\syz(J)),\deg(R/J)\bigr) = (e,j).    
\end{equation}
Each admissible pair $(e,j)$ will be called \emph{realizable} (for degree-vector $\textbf{d}$ and dimension $n$) whenever there is a three-generated ideal $J = (f_1, f_2, f_3) \subset k[x_1, \ldots, x_n]$ with $\deg(f_i) = d_i$ satisfying \eqref{eq:realizable-pairs}.  Theorem~\ref{thm:every-e-occurs-structurally} shows that every
numerically possible initial syzygy degree occurs. More precisely, for
every
\[
d_2<e\le d_1+d_2,
\]
the pair
\[
\bigl(e,(d_1+d_2-e)d_3\bigr)
\]
is realizable. However, not every admissible pair is realizable.
Remark~\ref{rem:universal-structural-gap} shows that the admissible pair
\[
(d_1+d_2,d_1d_2)
\]
never occurs. When one generator is linear,
Proposition~\ref{prop:linear-generator-structural-spectrum} gives the
complete realizable spectrum. Proposition~\ref{prop:top-boundary-structural-spectrum}
gives further information when the two largest generator degrees are
equal. These results lead to the general realization problem stated in
Question~\ref{ques:structural-spectrum}.

The equigenerated case is especially important because it includes
gradient ideals of plane curves. Suppose that $J$ is generated by three
forms of the same degree $d$, and let $e$ denote the standard degree of
the coefficients of an initial syzygy. Then the general formula becomes
\[
\Bour(J)+\deg(R/J)=d^2-ed+e^2.
\]
Theorem~\ref{BourBoundEquigenerated} proves, without the local-freeness
assumption used in the earlier result of~\cite[Theorem~3.3]{MFA}, that
\[
\Bour(J)\le e^2.
\]
Consequently,
\[
d(d-e)\le\deg(R/J)\le d^2-ed+e^2.
\]
These inequalities are the natural analogues, for arbitrary
three-generated equigenerated ideals, of the du Plessis--Wall bounds
for the global Tjurina number of a reduced plane curve
\cite[Theorem 3.2]{CTC-Plessis}. The bound $\Bour(J)\le e^2$ is sharp for every
$1\le e<d$, even among gradient ideals, as shown in
Example~\ref{Bour-e-square-gradient}. Proposition~\ref{prop:BourJ-max-res}
also characterizes the equality $\Bour(J)=e^2$ in terms of the minimal
graded free resolution.

The comparison with du Plessis and Wall is important for another
reason. The gap phenomenon discussed in
\cite[Section~4]{CTC-Plessis} is numerical. For each possible initial
degree of the syzygy module of a gradient ideal, their bounds give an
interval in which the global Tjurina number must lie. The gaps in their
table come from the fact that some of these intervals do not meet.
The same numerical picture appears for the bounds above. For
$1\le e<d$, the allowed intervals for $\deg(R/J)$ have total union
\[
[d,d^2-d+1]\cap\mathbb Z,
\]
while the formal block corresponding to $e=0$ consists only of $d^2$.
Thus, in this range, the numerical bounds leave the values
\[
G_d=\{d^2-b\mid1\le b\le d-2\}.
\]
If one of these values is realized, it must occur on the boundary
$e=d$.

The situation for the Bourbaki degree is different. The corresponding
numerical intervals are nested, and their union is
\[
[0,(d-1)^2]\cap\mathbb Z.
\]
Therefore the bounds themselves give no numerical gaps for the
Bourbaki degree in the range $e<d$.

This is where the realization problem begins. Numerical bounds tell us
which pairs are allowed, but they do not tell us which pairs actually
exist. The gap analysis of du Plessis and Wall does not address this
question inside an allowed interval. In the present setting, numerical
admissibility already fails to imply realizability for arbitrary
three-generated ideals. This motivates Definition~\ref{defi:gaps}.
A \emph{structural gap} is a numerically admissible pair which is not
realized in a fixed dimension. A \emph{dimensional gap} is a structural
gap which becomes realizable after increasing the dimension. Finally,
in three variables, an \emph{integrable gap} is a pair which is
realizable by an arbitrary triple but not by the gradient ideal of a
reduced plane curve. Thus we have the chain
\[
\text{numerically admissible}
\quad\supseteq\quad
\text{realizable by a triple}
\quad\supseteq\quad
\text{realizable by a gradient triple}.
\]
The numerical gaps in the du Plessis--Wall picture concern numerical
possibility, while structural and integrable gaps concern actual
realization.

To study structural gaps, we use the resolution theory of
three-generated almost Cohen--Macaulay ideals developed in~\cite{BFRS}.
In Section~\ref{sec:hom-dim-2} we consider $d$-equigenerated strict
almost complete intersections of height two and homological dimension
two. Theorem~\ref{low-degree-gap} gives two low-degree exclusions:
\[
\Bour(J)\neq2\quad\text{if }d=2,
\qquad
\Bour(J)\neq5\quad\text{if }d=3.
\]
These values are allowed by the numerical bounds, so their absence is
not numerical. In dimension three they give structural obstructions.
The quadratic case also shows that dimension matters: the corresponding
pair cannot occur in dimension three, but it is realized in dimension
four.

The cases $d=2$ and $d=3$ suggest that the value $(d-1)^2+1$ might be excluded in every degree. Proposition~\ref{prop:Jd-family}
shows that this is false. For every $d\ge4$, there exists an ideal
$J\subseteq k[x,y,z]$ such that
\[
\indeg(\syz(J))=d,\qquad
\deg(R/J)=2d-2,\qquad
\Bour(J)=(d-1)^2+1.
\]
Thus the exclusions in degrees two and three are special low-degree
phenomena. Example~\ref{ex:quintic-cuspidal} also shows, in degree
$d=4$, that the same Bourbaki value is attained by the gradient ideal
of a quintic curve.

The homological structure also gives a strong restriction near the
other end of the degree range. Theorem~\ref{structural-gap-degree-general}
shows that a $d$-equigenerated strict almost complete intersection of
homological dimension two satisfies
\[
\deg(R/J)\le d^2-d.
\]
Hence the numerically possible value $d^2-d+1$ cannot occur in this
homological class, although it does occur in the perfect case.
Theorem~\ref{structural-gap-degree-general-desc} describes the extremal
case $\deg(R/J)=d^2-d$: only two forms of the minimal graded free
resolution can occur, and in both cases
\[
\Bour(J)=1.
\]
This gives another example where the minimal resolution imposes a
restriction which is not visible from the general numerical bounds.

The resolution theory can be described completely in the first two
degrees. We first obtain a formula for the Bourbaki degree of ideals
with exactly three generating syzygies in
Proposition~\ref{Bourm=3}. We then determine all possible resolution
data for ideals generated by quadrics and by cubics in
Subsections~\ref{subsec:ideals-quadrics} and
\ref{subsec:ideals-cubics}. These classifications allow us to compare
arbitrary triples with gradient ideals. For quadratic generators, every
possible resolution in three variables occurs as the resolution of the
gradient ideal of a plane cubic. The same is true for cubic generators,
as follows from the study of quartic curves.

We therefore return to plane curves in
Section~\ref{sec:quartic-curves}. Let
$X=V(f)\subseteq\PP^2$ be a reduced singular quartic which is not a
cone. Then $J_f$ is generated by three cubics, with
$\deg(R/J_f)=\tau(X)$. If $e=\indeg(\syz(J_f))$, then
\[
\Bour(X)=9-3e+e^2-\tau(X).
\]
Using the classification of three-generated cubic ideals together with
Wall's classification of reduced quartics~\cite{Wall-quartics},
Theorem~\ref{Bour-Quartic} determines the Bourbaki degree from the singularity configuration, except for the two configurations $3A_1$
and $A_3$.The possible values
are
\[
0,1,2,3,4,6,7,8.
\]
In particular, the value $5$ never occurs and, except for the cases $3A_1$ and $A_3$, the singularity configuration alone determines the Bourbaki degree (see Remark~\ref{rmk:ambiguity-sing-type}).
In these two exceptional cases one has $\tau(X)=3$, and the
possible values of $\indeg(\syz(J_f))$ distinguish
$\Bour(X)=4$ from $\Bour(X)=6$.

The missing value $5$ is a useful example of the difference between a
numerical gap and a structural one. It is not excluded by the
generalized du Plessis--Wall bounds. Instead, its occurrence would
require the pair
\[
(\tau(X),e)=(4,3),
\]
and this pair cannot occur for a three-generated cubic ideal.
Thus the missing Bourbaki value $5$ comes from a structural restriction,
not from the numerical bounds.

Finally, comparison with
Subsection~\ref{subsec:ideals-cubics} shows that every possible minimal
graded free resolution of a three-generated cubic ideal in three
variables is realized by the gradient ideal of a quartic; see
Table~\ref{table:gradient-triples-quartics}. Together with the analogous
statement for plane cubics, this shows that there are no integrable gaps
for generator degrees $d=2,3$.

The results of the paper therefore separate three different questions:
what is numerically allowed, what can be realized by a three-generated
ideal, and what can be realized by a gradient ideal. The Bourbaki
formula gives a common numerical framework for these questions, while
the structural results show that numerical admissibility alone is not
enough. The comparison with plane curves then shows how the same
invariants interact with singularities and with the gradient condition.
Understanding these realization problems in higher degree and in higher dimensions remains a
natural question.
\subsection*{Acknowledgments}
MJ is partially supported by the CNPq grant number 305601/2022-9 and the FAPESP-CEPID Project 2024/00923-6.
FM is supported by the São Paulo Research Foundation (FAPESP) by the PhD grant number $\#2021/10550-4$, under the cotutelle supervision of Marcos Jardim and Daniele Faenzi.
ANN was partially supported by FAPESP Grant No.~2026/08480-1  and sincerely thanks IMECC–UNICAMP for its generous hospitality and for providing an excellent academic environment during his visit. ZR is partially supported by the CNPq grant number 304056/2026-0. 

\section{Bourbaki degree of three generated ideals}\label{sec:3-gen-general}
In this section, we extend the notion of Bourbaki degree from the case of gradient ideals of plane curves (established in \cite{MAA}) to arbitrary ideals minimally generated by three homogeneous forms.

Let $R=k[x_1,\ldots,x_n]$ be a polynomial ring with standard grading over an algebraically closed field $k$ of characteristic zero, with $n \geq 3$. Let $J=(f_1,f_2,f_3)\subseteq R$ be an ideal generated minimally by three forms of $\deg(f_i)=d_i$ with $1\leq d_1\leq d_2\leq d_3$. Assume that $\gcd(f_1,f_2,f_3)=1$, so that $\hht(J)\geq 2$. We write $\indeg(\syz(J))$ for the \emph{initial degree} of the syzygy module $\syz(J)$, the minimum integer $e$ such that the graded piece $\syz(J)_e \neq 0$.

Throughout the paper, for a homogeneous ideal $J\subseteq R$ with $\hht(J)\geq 2$, we write $\deg(R/J)$ for the codimension-two degree of $R/J$. More precisely, from the associativity formula (see, for example, \cite[Corollary 4.6.8]{bruns1998cohen}), one may set:
\[
\deg(R/J):=\sum_{\substack{\fp\in\Ass(R/J)\\ \hht(\fp)=2}}
\length_{R_\fp}\big((R/J)_\fp\big)\deg(R/\fp)
\]
where $\lambda$ denotes the length of the Artinian module over the local ring. In particular, if $\hht(J)=2$, $\deg(R/J)$ is the usual multiplicity of $R/J$, while if $\hht(J)\geq 3$, then $\deg(R/J)=0$.

Set
\[
D:=d_1+d_2+d_3,
\qquad
\sigma:=d_1d_2+d_1d_3+d_2d_3.
\]
One has the following graded exact sequence of $R$-modules
\[ 0\lar \syz(J)\lar \bigoplus_{i=1}^3R(-d_i)\xrightarrow{[f_1\  f_2\ f_3]} J\lar 0.\]
Let $\nu=(a_1,a_2,a_3)\in\syz(J)$ be a minimal homogeneous generator of shifted degree $e$, so that $\deg(a_i)=e-d_i$ whenever $a_i\neq0$. It induces a graded homomorphism $R(-e)\stackrel{\nu}{\lar} \syz(J)$ fitting into an exact sequence of graded $R$-modules
\[ 0\lar R(-e)\stackrel{\nu}{\lar} \syz(J) \lar M_\nu\lar 0, \]
where $M_\nu$ is the cokernel of the map $\nu$. 

The following elementary lemma is well-known; we include the proof for completeness.

\begin{Lemma}\label{lem:cok-torsion-free}
$M_\nu$ is a torsion-free graded $R$-module of rank one.
\end{Lemma}

\begin{proof}
Since $\nu=(a_1,\, a_2,\, a_3)$ is a \emph{minimal} homogeneous generator, the coordinates $a_1,a_2,a_3$ have no nonconstant common factor. Indeed, if $a_i=g b_i$ with $\deg g>0$, then from
$a_1f_1+a_2f_2+a_3f_3=0$ and the fact that $R$ is a domain, one gets
$b_1f_1+b_2f_2+b_3f_3=0$, so $(b_1,b_2,b_3)\in \syz(J)$ and
$\nu=g(b_1,b_2,b_3)\in \fm\syz(J)$, where $\fm$ is the irrelevant maximal ideal of $R$, contradicting the minimality of $\nu$.

Now let
\[
C:=\cok\!\Big(R(-e)\xrightarrow{\ \widetilde{\nu}\ }
\bigoplus_{i=1}^3R(-d_i)\Big)
\]
where $\widetilde{\nu}$ is the natural lift of $\nu$ via the inclusion $\syz(J) \hookrightarrow \bigoplus_{i=1}^3 R(-d_i)$. We claim that $C$ is a torsion-free module. Let $h\neq 0$ and suppose that
$h\bar u=0$ in $C$, where $\bar u$ is the class of $u=(u_1,u_2,u_3)$.
Then $hu=r(a_1,a_2,a_3)$ for some $r\in R$, i.e.,
$hu_i=ra_i$ for all $i$. Since $\gcd(a_1,a_2,a_3)=1$ and $R$ is a UFD,
it follows that $h\mid r$. Writing $r=hs$, we get
$hu_i=hsa_i$, hence $u_i=sa_i$ for all $i$, because $R$ is a domain.
Thus $u=s(a_1,a_2,a_3)$ lies in the image of $\widetilde{\nu}$, so
$\bar u=0$, and thus $C$ is torsion-free. By the snake lemma applied to the commutative diagram
\[
\begin{array}{ccccccccc}
0&\to&R(-e)&\xrightarrow{\nu}&\syz(J)&\to&M_\nu&\to&0\\
&&\Vert&&\downarrow&&\downarrow\\
0&\to&R(-e)&\xrightarrow{\widetilde{\nu}}&
\bigoplus_{i=1}^3R(-d_i)&\to&C&\to&0
\end{array}
\]
we obtain an injection $M_\nu\hookrightarrow C$. Hence $M_\nu$ is torsion-free. Since $\rk \syz(J)=3-\rk J=2$, it follows that $\rk(M_\nu)=1$. Therefore $M_\nu$ is torsion-free of rank one. 
\end{proof}

The following result may be viewed as an extension of \cite[Theorem 2.1]{MAA} to arbitrary three-generated homogeneous ideals.

\begin{Theorem}\label{Bour3Gens}
Let $\nu\in \syz(J)$ be a minimal homogeneous syzygy of degree $e\geq 1$. Then the following hold.
\begin{enumerate}
\item[\rm(a)]
$M_\nu$ is free if and only if $J$ is a perfect ideal of codimension $2$. In this case,
\[
\syz(J)\simeq R(-e)\oplus R(-(D-e)),
\]
hence $J$ admits a graded Hilbert-Burch minimal free resolution
\[
0\lar R(-e)\oplus R(-(D-e))
\lar
R(-d_1)\oplus R(-d_2)\oplus R(-d_3)
\lar J\lar 0,
\]
and
\[
\deg(R/J)=e^2-eD+\sigma.
\]

\item[\rm(b)]
If $J$ is not a perfect ideal of codimension $2$, then $M_\nu$ is isomorphic to an unmixed homogeneous ideal $I_\nu\subseteq R$ of height $2$, up to shift. More precisely,
\[
M_\nu \simeq I_\nu(e-D).
\]
Moreover,
\[
\deg(R/I_\nu)=e^2-eD+\sigma-\deg(R/J).
\]

\item[\rm(c)]
Assume that $J$ is not a perfect ideal of codimension $2$, and choose a complete set of minimal homogeneous
generators of $\syz(J)$ containing $\nu$. Let
\[
\cdots \to F_1
\xrightarrow{\left[\begin{smallmatrix}\lambda\\ \psi\end{smallmatrix}\right]}
R(-e)\oplus F_0
\xrightarrow{(\nu,\varphi)}
\syz(J)
\lar 0
\]
be the resulting graded minimal free resolution. Then a graded minimal free
resolution of $I_\nu$ is given by
\[
\cdots \to F_1(D-e)
\xrightarrow{\ \psi(D-e)\ }
F_0(D-e)
\lar I_\nu
\lar 0.
\]
\end{enumerate}
\end{Theorem}

\begin{proof}
\noindent
(a) Suppose first that $M_\nu$ is free. Since it has rank one, there exists an integer
$s$ such that $M_\nu\simeq R(s)$, and the short exact sequence becomes
\[
0\lar R(-e)\lar \syz(J)\lar R(s)\lar 0,
\]
which must split, hence
\[
\syz(J)\simeq R(-e)\oplus R(s)
\]
is free of rank two, and $J$ has projective dimension one. Hence $\pd_R(R/J)=2$. Since $R$ is regular,
\[
\hht(J)=\grade(J)\le\pd_R(R/J)=2.
\]
Together with the standing hypothesis $\hht(J)\ge2$, this gives
$\hht(J)=2$. Thus $R/J$ is Cohen--Macaulay, and $J$ is perfect
of height two.

Conversely, if $J$ is perfect of height two, then $\syz(J)$ is a free graded module of rank two.
Because $\nu$ is a minimal generator of $\syz(J)$, it can be completed to a homogeneous
basis, so its cokernel $M_\nu$ is free of rank one. It remains to determine the second shift. Write $M_\nu\simeq R(s)$. From the exact
sequences above, one obtains
\[
0\lar R(-e)\oplus R(s)
\stackrel{\phi=(a_{i,j})}\lar
R(-d_1)\oplus R(-d_2)\oplus R(-d_3)
\lar R\to R/J\lar 0.
\]
where $\deg a_{i,1}=e-d_i$ and $\deg a_{i,2}=-s-d_i$. By the Hilbert-Burch theorem, \[f_3=\det\begin{bmatrix}
    a_{1,1}&a_{1,2}\\
    a_{2,1}&a_{2,2}
\end{bmatrix}.\] Hence, 
\[d_3=\deg f_3=e-s-d_1-d_2,\]
equivalently,
$$s=e-D.$$
With this and the minimal graded free resolution of $R/J$, we conclude that the Hilbert series of $R/J$ is
\[
\Hs_{R/J}(t)=\frac{1-t^{d_1}-t^{d_2}-t^{d_3}+t^e+t^{D-e}}{(1-t)^{n}}.
\] Taking the second derivative of the numerator at $t=1$ and dividing by $2$, we get
\[
\deg(R/J)=e^2-eD+\sigma.
\]

\smallskip

\noindent
{\rm(b)} Assume that $J$ is not a perfect ideal of codimension $2$. Since $M_\nu$ is torsion-free of rank one and $R$ is a UFD, $M_\nu$ is isomorphic to
\[
M_\nu\simeq I_\nu(s)
\]
for some homogeneous ideal $I_\nu\subseteq R$ and some integer $s$. If $\hht(I_\nu)=1$, let $h\neq 0$ be the homogeneous gcd of a set of homogeneous generators of $I_\nu$. Then $I_\nu=hK$ for some homogeneous ideal $K\subseteq R$. By maximality of $h$, the ideal $K$
has no nonconstant common factor; hence $\hht(K)\ge 2$. Since multiplication by $h$ induces an isomorphism of $R$-modules $K\overset{\cdot h}{\simeq} I_\nu$ (and a graded isomorphism $K(-\deg h)\simeq I_\nu$), replacing $I_\nu$ by $K$ only changes the shift. Therefore, after this replacement, we may assume that
$\hht(I_\nu)\ge 2$.

We now prove that $I_{\nu}$ is unmixed of height exactly two. Let $\fp \in \Ass(R/I_\nu)$. Since
$\hht(I_\nu)\geq 2$, we have $\hht(\fp)\geq2$. We show that $\hht(\fp)$ cannot be
strictly larger than $2$.

Set $E=\syz(J)$. We claim that
\[
\depth E_\fp \geq2
\]
for every prime $\fp$ with $\hht(\fp)\geq3$. Indeed, if $\fp \not\supseteq J$, then
$J_\fp=R_\fp$, and the sequence
\begin{align*}
0\to E_\fp\to R_\fp^3\to R_\fp\to0    
\end{align*}
splits; hence $E_\fp$ is free. Otherwise, when $\fp\supseteq J$, then $J_\fp$ is a nonzero ideal
of the domain $R_\fp$, so $\depth J_\fp\geq1$. From $\depth R_\fp=\hht(\fp)\geq3$, using the depth lemma~\cite[Proposition~1.2.9]{bruns1998cohen} with respect to 
$$
0 \to E_{\fp} \to R_{\fp}^3 \to J_\fp \to 0
$$
yields $\depth E_\fp\geq2$.

Now localize
\[
0\to R(-e)\to E\to I_\nu(s)\to0
\]
at a prime $\fp$ with $\hht(\fp)\geq3$. Since $\depth R_\fp\geq3$ and
$\depth E_\fp\geq2$, the depth lemma gives
\[
\depth (I_\nu)_\fp \geq2.
\]
Using the exact sequence
\[
0\to (I_\nu)_\fp \to R_\fp \to (R/I_\nu)_\fp \to0,
\]
we obtain $\depth (R/I_\nu)_\fp \geq1$. Therefore no prime of height at least $3$ can be associated to $R/I_\nu$, so that every associated prime of $R/I_\nu$ has height $2$, and $I_\nu$ is unmixed
of height $2$.

To determine the shift and the degree, write
\[
\Hs_{R/J}(t)=\frac{h(t)}{(1-t)^{n-2}},
\qquad
\Hs_{R/I_\nu}(t)=\frac{h_{I_\nu}(t)}{(1-t)^{n-2}},
\]
where
\[
h(1)=\deg(R/J),
\qquad
h_{I_\nu}(1)=\deg(R/I_\nu).
\]
From
\[
\Hs_{M_\nu}(t)
=
\Hs_{\syz(J)}(t)-\frac{t^e}{(1-t)^n}
=
\frac{t^{d_1}+t^{d_2}+t^{d_3}-t^e-1}{(1-t)^n}
+\Hs_{R/J}(t),
\]
and from the isomorphism $M_\nu\simeq I_\nu(s)$ it follows that
\[
\Hs_{M_\nu}(t)=t^{-s}\Hs_{I_\nu}(t)
=
t^{-s}\!\left(\frac{1}{(1-t)^n}-\Hs_{R/I_\nu}(t)\right).
\]
Comparing the two expressions for $\Hs_{M_\nu}(t)$, we get
\[
t^{-s}-t^{d_1}-t^{d_2}-t^{d_3}+t^e+1
=
(1-t)^2\bigl(t^{-s}h_{I_\nu}(t)+h(t)\bigr).
\]
Taking the first derivative at $t=1$ yields $s=e-D$ and taking the second derivative at $t=1$ yields:
\[
\deg(R/I_\nu)+\deg(R/J)
=e^2-eD+\sigma.
\]

\smallskip

\noindent

For part~(c), start with a graded minimal free resolution of $\syz(J)$
adapted to the chosen generator $\nu$:
\[
\cdots \to F_1
\xrightarrow{\left[\begin{smallmatrix}\lambda\\ \psi\end{smallmatrix}\right]}
R(-e)\oplus F_0
\xrightarrow{(\nu,\varphi)}
\syz(J)
\lar 0.
\]
Since $M_\nu=\syz(J)/R(-e)\nu$, composing $(\nu,\varphi)$ with the quotient map
$\syz(J)\twoheadrightarrow M_\nu$ kills the summand $R(-e)$ and induces a surjection
\[
F_0 \lar M_\nu.
\]
Its kernel is precisely $\operatorname{Im}(\psi)$: indeed, an element
$y\in F_0$ maps to zero in $M_\nu$ if and only if $\varphi(y)\in R(-e)\nu$,
equivalently, if and only if there exists $x\in R(-e)$ such that
$(x,y)\in \ker(\nu,\varphi)=\operatorname{Im}\!\left[\begin{smallmatrix}\lambda\\ \psi\end{smallmatrix}\right]$,
which means exactly that $y\in \operatorname{Im}(\psi)$. Therefore, one obtains a resolution
\[
\cdots \to F_1
\xrightarrow{\ \psi\ }
F_0
\lar M_\nu
\lar 0.
\]
By part~(b), $M_\nu\simeq I_\nu(e-D)$. Shifting by $D-e$, we get a graded free resolution
\[
\cdots \to F_1(D-e)
\xrightarrow{\ \psi(D-e)\ }
F_0(D-e)
\lar I_\nu
\lar 0.
\]
This resolution is minimal because the original resolution of $\syz(J)$ is minimal;
hence all entries of the matrix of $\psi$ and the next maps lie in the irrelevant maximal ideal.
\end{proof}

\begin{Definition}\label{defi:Bourbaki-degree-non-equigen}
Assume that $e=\indeg(\syz(J))$, and let $\nu\in \syz(J)$ be any minimal generator of
degree $e$. If $J$ is not a perfect ideal of height two, let $I_\nu$ be the ideal appearing in 
Theorem~\ref{Bour3Gens}. We call $I_\nu$ the \emph{Bourbaki ideal} of $J$ associated to $\nu$. The \emph{Bourbaki degree} of $J$ is defined by:
$$
\Bour(J):= \begin{cases}
0,             &\text{ if }J\text{ is perfect of height two;}\\
\deg(R/I_\nu), &\text{ otherwise.}\\
\end{cases} .
$$
By Theorem~\ref{Bour3Gens}, this number is independent of the choice
of $\nu$ of initial degree; it satisfies
\begin{equation}\label{eq:Bourbaki-formula-non-equigen}
\Bour(J)=e^2-eD+\sigma-\deg(R/J)    
\end{equation}
and it is zero if and only if the syzygy module $\syz(J)$ is graded free.
\end{Definition}

We record a numerical consequence of the Bourbaki degree formula. This bound will be useful when comparing the Bourbaki degree with the degree of $R/J$.

\begin{Proposition}\label{prop:basic-numerical-bound}
Let $J=(f_1,f_2,f_3)\subseteq R=k[x_1,\ldots,x_n]$ be minimally generated by
homogeneous forms of degrees $1\le d_1\le d_2\le d_3$, with
$\gcd(f_1,f_2,f_3)=1$. Set $D=d_1+d_2+d_3$ and
$\sigma=d_1d_2+d_1d_3+d_2d_3$, and let $e=\indeg(\syz(J))$ in the shifted
grading. Then
\[
\Bour(J)\le e^2.
\]
Consequently,
\[
\sigma-eD\le \deg(R/J)\le e^2-eD+\sigma.
\]
\end{Proposition}

\begin{proof}
The right-hand inequality is clear as $\Bour(J) \geq 0$. For the left-hand inequality, by the formula \eqref{eq:Bourbaki-formula-non-equigen}, it suffices to show $
\deg(R/J)\ge \sigma-eD$, which by the same formula is equivalent to $\Bour(J) \leq e^2$. 

Let $\nu=(a_1,a_2,a_3)\in\syz(J)$ be a nonzero homogeneous syzygy of shifted
degree $e$. At least two of the coordinates of $\nu$ are nonzero. First, assuming that all three coordinates of $\nu$ are nonzero, we get
$e-d_i\ge0$ for every $i$, and hence $e\ge d_3$. Therefore
\[
eD\ge d_3(d_1+d_2+d_3)\ge d_1d_2+d_1d_3+d_2d_3=\sigma,
\]
because $d_3^2\ge d_1d_2$. Hence $\sigma-eD\le0$, and the desired inequality
follows from $\deg(R/J)\ge0$.

Now suppose that exactly two coordinates of $\nu$ are nonzero, say the $i$-th
and $j$-th coordinates, and let $\{i,j,k\}=\{1,2,3\}$. Put
$c=\deg\gcd(f_i,f_j)$. Then the primitive Koszul syzygy between $f_i$ and
$f_j$ has shifted degree $d_i+d_j-c$. Since $e$ is the initial syzygy degree,
we must have
\[
e=d_i+d_j-c.
\]
Write $f_i=gu$ and $f_j=gv$, where $g=\gcd(f_i,f_j)$ and $\deg g=c$. When $c = 0$, then $g$ is a unit, so we get $(g, f_k) = R$ and the needed inequality is simply $\deg(R/J) \geq 0$, so it follows. Thus, we may assume $c > 0$.

Since
$\gcd(f_1,f_2,f_3)=1$, we have $\gcd(g,f_k)=1$. Thus $(g,f_k)$ is a complete
intersection of degrees $c$ and $d_k$, and so
\[
\deg(R/(g,f_k))=cd_k.
\]
Moreover $J\subseteq (g,f_k)$, hence
\[
\deg(R/J)\ge \deg(R/(g, f_k)) = cd_k.
\]
On the other hand,
\[
\sigma-eD
=cd_k+d_id_j-(d_i+d_j-c)(d_i+d_j)\le cd_k,
\]
because $c\le\min\{d_i,d_j\}$. Thus
\[
\deg(R/J)\ge \sigma-eD.
\]
\end{proof}

The next result packages the basic numerical restrictions on the initial syzygy degree and the Bourbaki degree, together with the sharpness case and the fixed-gap degree patterns.
\begin{Proposition}\label{prop:numerical-bounds-non-equigenerated}
Let $J=(f_1,f_2,f_3)\subseteq R=k[x_1,\ldots,x_n]$ be minimally generated by homogeneous forms of degrees $1\le d_1\le d_2\le d_3$, with $\gcd(f_1,f_2,f_3)=1$, and let $e=\indeg(\syz(J))$ in the shifted grading. Set $c_{12}=\deg(\gcd(f_1,f_2))$. Then:
\begin{enumerate}
\item[\rm(i)] One has $d_2<e\le d_1+d_2-c_{12}\le d_1+d_2$. Moreover, if $e\le d_3$, then $e=d_1+d_2-c_{12}$.

\item[\rm(ii)] One has
\[
\Bour(J)\le e^2-eD+\sigma\le d_1d_3.
\]
If $d_2<d_3$, then the second inequality is strict; hence $\Bour(J)<d_1d_3$.

\item[\rm(iii)] The upper bound is sharp precisely in the complete intersection case with $d_2=d_3$. More precisely, $\Bour(J)=d_1d_3$ if and only if $d_2=d_3$ and $J$ is a complete intersection.

\item[\rm(iv)] If $d_1=1<d_2\le d_3$, then, after a linear change of coordinates and replacing the other two generators modulo $x_1$, $J=(x_1,g,h)$ with $g,h\in k[x_2,\ldots,x_n].$
Let $\ell=\gcd(g,h)$, and write $g=\ell g_1$, $h=\ell h_1$, with $\gcd(g_1,h_1)=1$. Then $R/J$ has a minimal graded free resolution
\small{\[
0\to R(-(d_2+d_3-\deg\ell+1))
\xrightarrow{\ \phi_2\ }
\begin{matrix}
R(-(d_2+d_3-\deg\ell))\\
\oplus\\
R(-(d_3+1))\\
\oplus\\
R(-(d_2+1))
\end{matrix}
\xrightarrow{\ \phi_1\ }
\begin{matrix}
R(-1)\\
\oplus\\
R(-d_2)\\
\oplus\\
R(-d_3)
\end{matrix}
\xrightarrow{\ \phi_0\ }
R\to R/J\to0.
\]}
In particular,
\[
\deg(R/J)=\deg\ell,
\qquad
e=d_2+1.
\]
Moreover, $e\le d_2+d_3-\deg(R/J)$. 

\item[\rm(v)] Suppose $(d_1,d_2,d_3)=(d,d+c_1,d+c_2)$, where $d\ge1$ and $0\le c_1\le c_2$. Write $e=d+c_1+s$. Then $1\le s\le d-c_{12}\le d$, and
\[
\Bour(J)+\deg(R/J)=d(d+c_2)-s(d+c_2-c_1-s).
\]
In particular,
\[
\Bour(J)\le d(d+c_2)-s(d+c_2-c_1-s)\le d(d+c_2).
\]
Moreover, if $s\le c_2-c_1$, then $s=d-c_{12}$.
\item[\rm(vi)] If $d_1=2\le d_2\le d_3$, then $e\in\{d_2+1,d_2+2\}$ and
\[
\Bour(J)+\deg(R/J)=
\begin{cases}
d_2+d_3-1, & \text{if } e=d_2+1,\\
2d_2, & \text{if } e=d_2+2.
\end{cases}
\]
If $d_3\ge d_2+1$, then $e=d_2+1$ if and only if $\deg(\gcd(f_1,f_2))=1$.
\end{enumerate}
\end{Proposition}

\begin{proof}
Let $\nu=(a_1,a_2,a_3)$ be a nonzero homogeneous syzygy of shifted degree $e$. At least two of the $a_i$ are nonzero, and for each nonzero $a_i$ one has $\deg(a_i)=e-d_i\ge0$. Hence $e\ge d_2$. Equality cannot occur: if $e=d_2$, then the syzygy has a nonzero constant coordinate on a generator of degree $d_2$ and no nonzero coordinate on any generator of larger degree, forcing one generator to lie in the ideal generated by the others, contrary to minimality. Thus $d_2<e$.

The Koszul relation between $f_1$ and $f_2$ has shifted degree $d_1+d_2-c_{12}$, so $e\le d_1+d_2-c_{12}$. If $e\le d_3$, then the third coordinate of any syzygy of shifted degree $e$ has degree $e-d_3\le0$. If it is a nonzero constant, then $f_3\in(f_1,f_2)$, contradicting minimality; hence it is zero. Thus the syzygy comes from the pair $(f_1,f_2)$, and $e=d_1+d_2-c_{12}$. This proves \rm(i).

Using the formula \eqref{eq:Bourbaki-formula-non-equigen} and since $\deg(R/J)\ge0$, one has $\Bour(J)\le e^2-eD+\sigma$. Set $F(t)=t^2-tD+\sigma$, so that
$$
F(t) - d_1 d_3 = (t-d_2)(t-d_1-d_3).
$$
By part $\mathrm{(i)}$, $e-d_2 > 0$, $e-d_1-d_3 \leq d_2 - d_3 \leq 0$, therefore $F(e) \leq d_1d_3$. Whenever $d_2 < d_3$, the second factor is strictly negative, so $F(e) < d_1 d_3$, and this shows \rm(ii).

For \rm(iii), if $d_2=d_3$ and $J$ is a complete intersection, then $\deg(R/J)=0$ by our convention and $e=d_1+d_2$, so $\Bour(J)=e^2-eD+\sigma=d_1d_2=d_1d_3$. Conversely, if $\Bour(J)=d_1d_3$, then equality holds in $\Bour(J)\le F(e)\le d_1d_3$. By \rm(ii), this forces $d_2=d_3$ and $\deg(R/J)=0$. Since $\gcd(f_1,f_2,f_3)=1$, our degree convention gives $\hht(J)=3$. As $J$ is generated by three elements in the Cohen--Macaulay ring $R$, it is a complete intersection.

For {\rm(iv)}, consider
\[
\phi_0=\begin{pmatrix}x_1&g&h\end{pmatrix}, \qquad \phi_1=
\begin{pmatrix}
0 & h & -g\\
-h_1 & 0 & x_1\\
g_1 & -x_1 & 0
\end{pmatrix},
\qquad
\phi_2=
\begin{pmatrix}
x_1\\
g_1\\
h_1
\end{pmatrix}.
\]
One checks directly that $\phi_0\phi_1=0$ and $\phi_1\phi_2=0$. Thus we have a
complex
\[
0\to R(-(d_2+d_3-\deg\ell+1))
\xrightarrow{\ \phi_2\ }
\begin{matrix}
R(-(d_2+d_3-\deg\ell))\\
\oplus\\
R(-(d_3+1))\\
\oplus\\
R(-(d_2+1))
\end{matrix}
\xrightarrow{\ \phi_1\ }
\begin{matrix}
R(-1)\\
\oplus\\
R(-d_2)\\
\oplus\\
R(-d_3)
\end{matrix}
\xrightarrow{\ \phi_0\ }
R\to R/J\to0.
\]
It is straightforward to check that 
$$
\hht(I_1(\varphi_0)) \geq 1, \quad \hht(I_2(\varphi_1)) \geq 2, \quad \hht(I_1(\varphi_2)) \geq 3.
$$
Thus the complex is exact by the Buchsbaum--Eisenbud acyclicity criterion (see, for example, \cite[Theorem~1.4.13]{bruns1998cohen}); it is minimal because all entries have positive degree.

The Hilbert numerator of $R/J$ is therefore
\[
1-t-t^{d_2}-t^{d_3}
+t^{d_2+d_3-\deg\ell}+t^{d_3+1}+t^{d_2+1}
-t^{d_2+d_3-\deg\ell+1}.
\]
Taking the second derivative at $t=1$ and dividing by $2$, we obtain $\deg(R/J)=\deg\ell$.
The first syzygy degrees are
\[
d_2+d_3-\deg\ell,\qquad d_3+1,\qquad d_2+1.
\]
Since minimality of $J$ gives $\deg\ell<d_2$, the smallest of these degrees is
$d_2+1$. Hence
\[
e=\indeg(\syz(J))=d_2+1.
\]
Finally,
\[
e=d_2+1\le d_2+d_3-\deg\ell
=d_2+d_3-\deg(R/J),
\]
because $\deg\ell<d_2\le d_3$.

For \rm(v), write $d_1=d$, $d_2=d+c_1$, $d_3=d+c_2$, and $e=d+c_1+s$. By \rm(i), $1\le s\le d-c_{12}\le d$. Since $D=3d+c_1+c_2$ and $\sigma=3d^2+2d(c_1+c_2)+c_1c_2$, substituting $e=d+c_1+s$ into $e^2-eD+\sigma$ gives
\[
e^2-eD+\sigma=d(d+c_2)-s(d+c_2-c_1-s).
\]
The displayed equality follows from Theorem~\ref{Bour3Gens}. The inequalities follow from $\deg(R/J)\ge0$ and item \rm(ii). Finally, if $s\le c_2-c_1$, then $e=d+c_1+s\le d+c_2=d_3$, so \rm(i) gives $e=d_1+d_2-c_{12}=2d+c_1-c_{12}$, equivalently $s=d-c_{12}$.

Finally, if $d_1=2$, then item~\rm(i) gives $d_2<e\le d_2+2$, hence $e=d_2+1$ or $e=d_2+2$. Substituting these two values into $e^2-eD+\sigma$, with $D=2+d_2+d_3$ and $\sigma=2d_2+2d_3+d_2d_3$, gives the displayed formula. If $d_3\ge d_2+1$, then $e=d_2+1\le d_3$, so item~\rm(i) gives $e=d_1+d_2-c_{12}=d_2+2-c_{12}$. Hence $e=d_2+1$ if and only if $c_{12}=1$, that is, if and only if $\deg(\gcd(f_1,f_2))=1$.
\end{proof}

The following proposition extends the characterization of small and maximal Bourbaki degree in \cite[Theorem 3.2]{MFA} to the non-equigenerated context.

\begin{Proposition}\label{prop:not-equigen-Bour-consequences}
Let $J=(f_1,f_2,f_3)\subseteq R=k[x_1,\ldots,x_n]$ be minimally generated by homogeneous forms of degrees $1\le d_1\le d_2\le d_3$, with $\gcd(f_1,f_2,f_3)=1$, and let $e=\indeg(\syz(J))$ in the shifted grading. Let $\nu$ be a minimal syzygy of degree $e$, and if $J$ is not perfect of codimension $2$, let $I_\nu$ be the corresponding Bourbaki ideal. Then:
\begin{enumerate}
\item[\rm(i)] $\Bour(J)=0$ if and only if $J$ is a perfect ideal of codimension $2$.
\item[\rm(ii)] $\Bour(J)=1$ if and only if $I_\nu$ is a complete intersection of two linear forms.
\item[\rm(iii)] $\Bour(J)=2$ if and only if $I_\nu$ is one of the following:
\begin{enumerate}
\item[\rm(a)] $I_\nu=\fp_1\cap\fp_2$, where $\fp_1,\fp_2$ are distinct height-two linear prime ideals;
\item[\rm(b)] $I_\nu=\fp$ is a height-two prime ideal with $\deg(R/\fp)=2$; equivalently, after a linear change of coordinates, $I_\nu=(\ell,q)$, where $\ell$ is linear and $q$ is quadratic, with $\bar q$ irreducible in $R/(\ell)$;
\item[\rm(c)] $I_\nu$ is $\fp$-primary for some height-two linear prime $\fp$ and $\length((R/I_\nu)_\fp)=2$.
\end{enumerate}
\item[\rm(iv)] If $d_2=d_3=d$, then $\Bour(J)=d_1d$ if and only if $J$ is a complete intersection.
\item[\rm(v)] If $d_2=d_3=d$, then $\Bour(J)=d_1d-1$ if and only if $e=d_1+d$ and $\deg(R/J)=1$.
\end{enumerate}
\end{Proposition}

\begin{proof}
Assertion~{\rm(i)} follows directly from
Definition~\ref{defi:Bourbaki-degree-non-equigen} and
Theorem~\ref{Bour3Gens}. 

For (ii) and (iii), assume $J$ is not perfect of height two. By Theorem~\ref{Bour3Gens}{\rm(b)}, the ideal $I_\nu$ is unmixed of height $2$ and $\Bour(J)=\deg(R/I_\nu)$. The assertions are then exactly the standard classification of unmixed codimension-two ideals of degree $1$ and $2$; see \cite[Theorem~3.2]{MFA}. The claim in (iv) follows from Proposition~\ref{prop:numerical-bounds-non-equigenerated} item (iii).

It remains to prove (v). Assume $d_2=d_3=d$, and again we set 
\begin{align*}
F(t)&= t^2 - t D + \sigma = t^2-t(d_1+2d)+2d_1d+d^2,
\end{align*}
so that $F(t) - d_1d = (t-d)(t-d_1-d)$. Since $d < e \leq d + d_1$, one has $F(e) \leq d_1 d$, with equality if and only if $e = d_1+d$. If $e=d_1+d$ and $\deg(R/J)=1$, then $F(e)=d_1d$ and hence $\Bour(J)=d_1d-1$. Conversely, assume $\Bour(J)=d_1d-1$. Since $F(e)\le d_1d$, we get $\deg(R/J)=F(e)-\Bour(J)\le1$. If $\deg(R/J)=0$, then the argument in (iv) gives that $J$ is a complete intersection, and thus $\Bour(J)=d_1d$, a contradiction. Thus $\deg(R/J)=1$. It follows that $F(e)=d_1d$, and therefore $e=d_1+d$.
\end{proof}

The case {\rm(v)} generalizes \cite[Theorem~3.8]{MAA} from equigenerated gradient ideals to the present non-equigenerated setting. In particular, when $d_1=d_2=d_3$, $n=3$ and $J_F=(\partial_1F,\partial_2F,\partial_3F)$ is a gradient ideal of an irreducible projective plane curve $V(F) \subseteq \PP^2$, item {\rm(v)} occurs precisely whenever $V(F)$ has a unique nodal singularity. 

The following construction shows that the value $\Bour(J) = d_1 d - 1$ is attained for every degree type $\textbf{d} = (d_1, d, d)$ with $2 \leq d_1 < d$, in fact by a non-empty Zariski open family of three-generated ideals.

\begin{Example}\label{ex:next-extremal-non-equigenerated}
		Let $R=k[x_1,\ldots,x_n]$ be a standard graded polynomial ring over a field $k$ with $n\geq 3$. For integers $2\leq d_1<d$ let $J\subseteq R$ be the ideal generated by the forms 
		$$f_1=A_1x_1-B_1x_2,\quad f_2=A_2x_1-B_2x_2,\quad f_3=A_3x_1-B_3x_2$$
		 for general forms $A_1,B_1\in R_{d_1-1}$ and $A_2,A_3,B_2,B_3\in R_{d-1}.$ We claim that the minimal graded free resolution of $R/J$ is
		\begin{equation}\label{resd1d}
		0\to R(-(2d+d_1-1))^2
		\xrightarrow{\ \Psi\ }
		\begin{matrix}
			R(-2d)\\
            \oplus\\
            R(-(d_1+d))^2\\
            \oplus \\
            R(-(2d+d_1-2))
		\end{matrix}
		\xrightarrow{\ \Phi\ }
		\begin{matrix}
			R(-d_1)\\
			\oplus\\
			R(-d)^2
		\end{matrix}
		\xrightarrow{\ [f_1\ f_2\ f_3]\ }
		R\to R/J\to0.
		\end{equation}
		where
		\[\Phi=\begin{bmatrix}
			0&-f_3&-f_2&A_3B_2-A_2B_3\\
			-f_3&0&f_1&-A_3B_1+A_1B_3\\
			f_2&f_1&0&A_2B_1-A_1B_2
		\end{bmatrix}\quad\mbox{and}\quad
\Psi=	\begin{bmatrix}
		-A_1&-B_1\\
		A_2&B_2\\
		-A_3&-B_3\\
		x_2&x_1
	\end{bmatrix}
	 \]
	 For this, consider the polynomial ring $T=R[y_1,y_2,y_3,z_1,z_2,z_3]$ with the grading  $$\deg y_1=\deg z_1=d_1-1,\quad \deg y_2=\deg y_3=\deg z_2=\deg z_3=d-1.$$  Let $\mathbb{J}$ be the height two ideal  of $T$ generated by 
	 $$h_1=x_1y_1-x_2z_1,\qquad
h_2=x_1y_2-x_2z_2,\qquad
h_3=x_1y_3-x_2z_3.$$ A straightforward calculation gives that the sequence 
	 \begin{equation}\label{bbJ}
	 	0\to T^2\stackrel{\begin{bmatrix}
	 		-y_1&-z_1\\
	 		y_2&z_2\\
	 		-y_3&-z_3\\
	 		x_2&x_1
	 \end{bmatrix}}{\longrightarrow} T^4\stackrel{\begin{bmatrix}
	 0&-h_3&-h_2&y_3z_2-y_2z_3\\
	 -h_3&0&h_1&-y_3z_1+y_1z_3\\
	 h_2&h_1&0&y_2z_1-y_1z_2
 \end{bmatrix}}{\longrightarrow} T^3\to T\to T/\mathbb{J}\to 0
\end{equation}
is a complex.  On the other hand, by the Buchsbaum-Eisenbud acyclicity theorem,  we have that this complex is a free resolution of $T/\mathbb{J}.$ Thus, the height of each associated prime of $T/\mathbb{J}$ is at most $3.$ With this and the fact that $(y_1,R_{d_1-1})$ is a $(d_1-1)$-equigenerated ideal of height $n+1\geq 4$ we have by the homogeneous prime avoidance lemma that for a general $A_1\in R_{d_1-1},$ the element $y_1-A_1$ is $T/\mathbb{J}$-regular. By induction, we can prove that $y_1-A_1,y_2-A_2,y_3-A_3,z_1-B_1,z_2-B_2,z_3-B_3$ is a $T/\mathbb{J}$-regular sequence. Thus,  tensoring \eqref{bbJ} by $T/(y_1-A_1,y_2-A_2,y_3-A_3,z_1-B_1,z_2-B_2,z_3-B_3)\simeq R$ we have that \eqref{resd1d} is a graded free resolution of $R/J,$ which is minimal because the entries of $\Phi$ and $\Psi$ are homogeneous polynomials of positive degree.

It follows from the minimal graded  free resolution of $R/J$ that the first syzygy degrees are
\[
d_1+d,\qquad d_1+d,\qquad 2d,\qquad 2d+d_1-2.
\]
Since \(2\le d_1<d\), the smallest of these is \(d_1+d\). Hence
\[
{\rm indeg}(\syz(J))=d_1+d.
\]
Furthermore, the Hilbert series numerator of \(R/J\) is
\[
1-t^{d_1}-2t^d+2t^{d_1+d}+t^{2d}+t^{2d+d_1-2}-2t^{2d+d_1-1}.
\]
Taking the second derivative at \(t=1\) and dividing by \(2\), we get
\[
\deg(R/J)=1.
\]
By Proposition~\ref{prop:not-equigen-Bour-consequences}{\rm(v)}, it follows
that ${\rm Bour}(J)=d_1d-1$. Thus the value $d_1d-1$ occurs for every $2\le d_1<d$.
		\end{Example}

\begin{Remark}\label{rem:syz-locally-free-and-Bourbaki-CM}
Let $R=k[x_1,\ldots,x_n]$ and let $J=(f_1,f_2,f_3)\subseteq R$ be minimally generated by homogeneous forms, with $\gcd(f_1,f_2,f_3)=1$.
\begin{enumerate}
\item[\rm(i)] The module $\syz(J)$ is locally free on $\Spec(R)\setminus\{\fm\}$ if and only if, for every prime $\fp \neq\fm$ containing $J$, the ideal $J_\fp$ has height $2$ and $R_\fp/J_\fp$ is Cohen--Macaulay. Indeed, if $\fp \not\supset J$, then $J_\fp=R_\fp$, and hence $\syz(J)_\fp$ is free. Now let $\fp \supset J$. Since $\gcd(f_1,f_2,f_3)=1$, one has $\height(J_\fp)\geq2$. Localizing the presentation of $J$ gives
\[
0\lar\syz(J)_\fp\lar R_\fp^3\lar J_\fp\lar0,
\]
so $\syz(J)_\fp$ is free if and only if $\pd_{R_\fp}(J_\fp)\leq1$, or equivalently $\pd_{R_\fp}(R_\fp/J_\fp)\leq2$. Since $R_\fp$ is regular, the Auslander--Buchsbaum formula, together with $\height(J_\fp)\geq2$, shows that this occurs if and only if $J_\fp$ has height $2$ and $R_\fp/J_\fp$ is Cohen--Macaulay.

\item[\rm(ii)] Keep the notation of Theorem~\ref{Bour3Gens}, and assume that $J$ is not a perfect ideal of codimension $2$. Let $I_\nu$ be the Bourbaki ideal determined by
\[
M_\nu=\syz(J)/R(-e)\nu\simeq I_\nu(e-D).
\]
If $\syz(J)$ is locally free on $\Spec(R)\setminus\{\fm\}$, then $R/I_\nu$ is locally Cohen--Macaulay on the punctured spectrum. This is the same argument as in \cite[Proposition~2.5]{MFA}. Indeed, for every prime $\fp \neq\fm$ containing $I_\nu$, localization of the Bourbaki sequence gives
\[
0\lar R_\fp\lar\syz(J)_\fp\lar(I_\nu)_\fp\lar0.
\]
Since $\syz(J)_\fp$ is free, it follows that $\pd_{R_\fp}((I_\nu)_\fp)\leq1$, and hence $\pd_{R_\fp}(R_\fp/(I_\nu)_\fp)\leq2$. Moreover, $I_\nu$ is unmixed of height $2$ by Theorem~\ref{Bour3Gens}{\rm(b)}. Thus $(I_\nu)_\fp$ has height $2$, and the Auslander--Buchsbaum formula shows that $R_\fp/(I_\nu)_\fp$ is Cohen--Macaulay.
\end{enumerate}
\end{Remark}

The following example shows that the conclusion of the Remark~\ref{rem:syz-locally-free-and-Bourbaki-CM} is genuinely local: even when $\syz(J)$ is locally free on the punctured spectrum, the associated Bourbaki quotient need not be Cohen--Macaulay.

\begin{Example}\label{ex:Bourbaki-skew-lines-not-CM}
Let $R=k[x_0,x_1,x_2,x_3]$, with $\operatorname{char}k=0$, and set
\[
f_1=x_0x_2+x_1x_3,\qquad
f_2=x_0x_3+x_1x_2,\qquad
f_3=x_0x_2+x_0x_3+x_1x_2.
\]
Let $J=(f_1,f_2,f_3)$. Then $J$ is minimally generated by three quadrics and
$\gcd(f_1,f_2,f_3)=1$. A direct check gives the following minimal graded free
resolution of $R/J$:
\[
0\to R(-6)\xrightarrow{\ \gamma\ }R(-5)^4\xrightarrow{\ \beta\ }R(-4)^5
\xrightarrow{\ \varphi\ }R(-2)^3\xrightarrow{\ [\,f_1\ f_2\ f_3\,]\ }R
\to R/J\to0,
\]
where
\[
\varphi=
\begin{pmatrix}
-f_2 & -x_2^2 & -f_2 & -f_3 & -x_0^2\\
f_1 & -x_2^2+x_2x_3+x_3^2 & -x_0x_3-x_1x_2+x_1x_3 & 0
& -x_0^2+x_0x_1+x_1^2\\
0 & x_2^2-x_3^2 & f_2 & f_1 & x_0^2-x_1^2
\end{pmatrix}.
\]
Let $\nu$ be the first column of $\varphi$, namely
\[
\nu=(-f_2,f_1,0).
\]
Thus $\nu$ is a minimal syzygy of shifted degree $e=4$. Since here
$d_1=d_2=d_3=2$, we have $D=6$. By Theorem~\ref{Bour3Gens}{\rm(c)}, the Bourbaki ideal $I_\nu$ has a graded free resolution of the form 
\[
0\to R(-4)\xrightarrow{\ \gamma\ }R(-3)^4\xrightarrow{\ \Phi\ }R(-2)^4
\to I_\nu\to0,
\]
where
\[
\Phi=
\begin{pmatrix}
x_1 & -x_0 & 0 & 0\\
-x_2 & -x_3 & -x_0 & x_1\\
x_3 & x_2 & x_1 & -x_0\\
0 & 0 & x_3 & x_2
\end{pmatrix}.
\]
We now identify this ideal. The row vector
\[
(h_1,h_2,h_3,h_4)
=
\bigl(
x_2^2-x_3^2,\,
x_0x_3+x_1x_2,\,
x_0x_2+x_1x_3,\,
x_0^2-x_1^2
\bigr)
\]
satisfies
\[
(h_1,h_2,h_3,h_4)\Phi=0.
\]
Therefore the above resolution identifies $I_\nu$ with the ideal $I=(h_1,h_2,h_3,h_4)$. After the linear change of variables
\[
x=x_0+x_1,\qquad y=x_0-x_1,\qquad z=x_2+x_3,\qquad w=x_2-x_3,
\]
we have
\[
h_1=zw,\qquad h_4=xy,\qquad 2h_2=xz-yw,\qquad 2h_3=xz+yw.
\]
Since $\operatorname{char}k=0$, it follows that
\[
I_\nu=I=(x,w)\cap(y,z).
\]
Thus $I_\nu$ is the defining ideal of the union of the two skew lines
$V(x,w)$ and $V(y,z)$ in $\PP^3$.

In particular, $R/I_\nu$ is locally Cohen--Macaulay on
$\Spec(R)\setminus\{\fm\}$, but it is not Cohen--Macaulay. 
\end{Example}

\subsection{Admissible and realizable pairs}

We now use the Bourbaki formula to organize the possible pairs
\[
\bigl(\indeg(\syz(J)),\deg(R/J)\bigr),
\]
where the initial degree of $\syz(J)$ is taken in the shifted grading. Fix
integers $1\le d_1\le d_2\le d_3$, and set
\[
\mathbf d=(d_1,d_2,d_3),\qquad
D=d_1+d_2+d_3,\qquad
\sigma=d_1d_2+d_1d_3+d_2d_3.
\]
For an integer $e$, define
\[
F_{\mathbf d}(e)=e^2-eD+\sigma.
\]
Thus, if $J=(f_1,f_2,f_3)$ has degree vector $\mathbf d$ and
$e=\indeg(\syz(J))$, then Theorem~\ref{Bour3Gens} gives
\[
\Bour(J)+\deg(R/J)=F_{\mathbf d}(e).
\]

By Proposition~\ref{prop:numerical-bounds-non-equigenerated}{\rm(i)}, every
possible initial syzygy degree satisfies
\[
d_2<e\le d_1+d_2.
\]
Moreover, Proposition~\ref{prop:basic-numerical-bound} gives
\[
\sigma-eD\le\deg(R/J)\le F_{\mathbf d}(e).
\]
Since $\deg(R/J)\ge0$, it follows that, for each fixed
$d_2<e\le d_1+d_2$, one has
\[
\deg(R/J)\in I_{e,\mathbf d}:=
\bigl[\max\{\sigma-eD,0\},F_{\mathbf d}(e)\bigr]\cap\mathbb Z.
\]

A pair of integers $(e,j) \in \mathbb{Z}^2$ will be called \emph{admissible} for $\mathbf d$ if
\[
d_2<e\le d_1+d_2
\qquad\text{and}\qquad
j\in I_{e,\mathbf d}.
\]
Such a pair will be called \emph{realizable} over
$R=k[x_1,\ldots,x_n]$ if there exists an ideal
$J=(f_1,f_2,f_3)\subseteq R$, minimally generated by homogeneous forms with
$\deg(f_i)=d_i$ and $\gcd(f_1,f_2,f_3)=1$, such that
\[
\indeg(\syz(J))=e
\qquad\text{and}\qquad
\deg(R/J)=j.
\]
The next result shows that every integer in the numerically permitted range
$d_2<e\le d_1+d_2$ occurs as an initial syzygy degree by exhibiting one
realizable pair for each such $e$.

\begin{Theorem}\label{thm:every-e-occurs-structurally}
Assume $n\ge3$. Let $\mathbf d=(d_1,d_2,d_3)$ with
$1\le d_1\le d_2\le d_3$. For every integer
\[
d_2<e\le d_1+d_2,
\]
the pair $(e, (d_1 + d_2 - e)d_3)$ is realizable.
\end{Theorem}

\begin{proof}
Fix $0\le c<d_1$ and set $e=d_1+d_2-c$. Write
$a=d_1-c$, $b=d_2-c$, $r=d_3$. Let $L=x_1+x_2+x_3$ and consider
\[
J=(f_1,f_2,f_3)\subseteq R,\qquad
f_1=x_1^cx_2^a,\quad f_2=x_1^cx_3^b,\quad f_3=L^r.
\]
Then $\deg(f_i)=d_i$, $\gcd(f_1,f_2,f_3)=1$, and the three forms minimally
generate $J$. We claim that $R/J$ has the following minimal graded free resolution:
\begin{equation}\label{res-initial}
0\to R(-(d_1+d_2+d_3-c))\stackrel{\phi_2}\longrightarrow\begin{matrix} R(-(d_2+d_3))\\\oplus \\R(-(d_1+d_3))\\\oplus \\R(-(d_1+d_2-c))\end{matrix}\stackrel{\phi_1}\longrightarrow\begin{matrix}R(-d_1)\\\oplus\\ R(-d_2)\\\oplus\\ R(-d_3)\end{matrix}\stackrel{\phi_0}\longrightarrow R\to R/J\to0	
\end{equation}
where
\[
\phi_0=\begin{pmatrix}f_1&f_2&f_3\end{pmatrix}, \qquad \phi_1=
\begin{pmatrix}
0 & L^r & -x_3^b\\
 L^r & 0 & x_2^a\\
-x_1^cx_3^b & -x_1^cx_2^a & 0
\end{pmatrix},
\qquad
\phi_2=
\begin{pmatrix}
x_2^a\\
-x_3^b\\
- L^r
\end{pmatrix}.
\]
Indeed, one checks directly that $\phi_0\phi_1=0$ and $\phi_1\phi_2=0$. Moreover, 
$$
\hht(I_1(\varphi_0)) \geq 1, \quad \hht(I_2(\varphi_1)) \geq 2, \quad \hht(I_1(\varphi_2)) \geq 3,
$$
so the complex is exact by the Buchsbaum--Eisenbud acyclicity criterion
\cite[Theorem~1.4.13]{bruns1998cohen}.
From this resolution,
\[
\Hs_{R/J}(t)=
\frac{
1-t^{d_1}-t^{d_2}-t^{d_3}
+t^{d_2+d_3}+t^{d_1+d_3}+t^{d_1+d_2-c}
-t^{d_1+d_2+d_3-c}
}{(1-t)^n}.
\]
The degree in codimension two is obtained from the second derivative of the numerator
at $t=1$. Thus
\[
\deg(R/J)=cd_3,
\]
and thus in the case $c = 0$ the ideal has height $3$.
The same resolution shows that the first syzygies have shifted degrees
\[
d_2+d_3,\qquad d_1+d_3,\qquad d_1+d_2-c.
\]
Since $d_3\ge d_2$ and $c<d_1$, the smallest of these is
$d_1+d_2-c$. Therefore
\[
\indeg(\syz(J))=d_1+d_2-c.
\]

Now set $e=d_1+d_2-c$. Then
\[
e=\indeg(\syz(J)),\qquad \deg(R/J)=cd_3.
\]
By Theorem~\ref{Bour3Gens},
$\Bour(J)+cd_3=F_{\mathbf d}(e)$. A direct computation gives
\[
F_{\mathbf d}(d_1+d_2-c)=(d_1-c)(d_2-c)+cd_3.
\]
Therefore
\[
\Bour(J)=(d_1-c)(d_2-c).
\]
Equivalently, since $c=d_1+d_2-e$, one has
\[
\Bour(J)=(e-d_1)(e-d_2),
\qquad
\deg(R/J)=(d_1+d_2-e)d_3.
\]
\end{proof}

\begin{Remark}\label{rem:universal-structural-gap}
Theorem~\ref{thm:every-e-occurs-structurally} shows that every integer in the
numerically permitted range for $e$ occurs as an initial syzygy degree. It
does not, however, imply that every admissible pair is realizable. In fact,
the upper endpoint of the numerical interval corresponding to
$e=d_1+d_2$ is never realized. Indeed, $F_{\mathbf d}(d_1+d_2)=d_1d_2$
and
\[
\sigma-(d_1+d_2)D=d_1d_2-(d_1+d_2)^2<0.
\]
Consequently,
\[
I_{d_1+d_2,\mathbf d}=[0,d_1d_2]\cap\mathbb Z,
\]
so the pair $(d_1+d_2,d_1d_2)$ is admissible. We claim that it is not
realizable.

Suppose, to the contrary, that there exists a minimally generated ideal
$J=(f_1,f_2,f_3)$ with degree vector $\mathbf d$ such that
\[
\indeg(\syz(J))=d_1+d_2
\qquad\text{and}\qquad
\deg(R/J)=d_1d_2.
\]
The Bourbaki formula then gives
\[
\Bour(J)
=F_{\mathbf d}(d_1+d_2)-\deg(R/J)
=0.
\]
By Proposition~\ref{prop:not-equigen-Bour-consequences}{\rm(i)}, the ideal
$J$ is perfect of codimension $2$. Hence
Theorem~\ref{Bour3Gens}{\rm(a)} gives
\[
\syz(J)\simeq
R(-(d_1+d_2))\oplus R(-d_3).
\]
Since $d_1+d_2$ is the initial syzygy degree, one must have
$d_1+d_2\le d_3$.

Consider a Hilbert--Burch matrix of $J$ whose column degrees are
$d_1+d_2$ and $d_3$. The two entries in the row corresponding to the
generator of degree $d_3$ have degrees
\[
d_1+d_2-d_3\le0
\qquad\text{and}\qquad
d_3-d_3=0.
\]
Since every non-zero entry of a minimal Hilbert--Burch matrix has positive degree,
both entries must vanish. This row is therefore zero, forcing two of the
maximal minors to vanish, contrary to the minimal generation of $J$ by
three nonzero forms. Thus the admissible pair
\[
(d_1+d_2,d_1d_2)
\]
is never realizable.
\end{Remark}

The first complete structural computation occurs when one generator is linear.
\begin{Proposition}\label{prop:linear-generator-structural-spectrum}
Assume $\mathbf d=(1,d_2,d_3)$ with $1<d_2\le d_3$ and $n\ge3$. Then every
ideal $J$ with degree vector $\mathbf d$ satisfies $\indeg(\syz(J))=d_2+1$ and
$I_{d_2+1,\mathbf d}=[0,d_2]\cap\mathbb Z$. Moreover, the admissible pair $(d_2+1,j)$ is realizable if and only if $0\le j\le d_2-1$. Thus $j=d_2$ is the unique nonrealizable value in this interval.
\end{Proposition}

\begin{proof}
Let $J=(f_1,f_2,f_3)$ have degrees $(1,d_2,d_3)$. After a linear change of
coordinates, write $J=(x_1,g,h)$ and set $S=R/(x_1)$. Let $\bar g,\bar h$ be
the images of $g,h$ in $S$. By
Proposition~\ref{prop:numerical-bounds-non-equigenerated}{\rm(iv)}, one has$
\indeg(\syz(J))=d_2+1$ and $\deg(R/J)=\deg\gcd(\bar g,\bar h)$. 
The Bourbaki formula gives
\[
\Bour(J)+\deg(R/J)
=F_{(1,d_2,d_3)}(d_2+1)=d_2.
\]
Moreover $\sigma-(d_2+1)D=-(d_2^2+d_2+1)<0$. Therefore
\[
I_{d_2+1,\mathbf d}=[0,d_2]\cap\mathbb Z.
\]
Set $t=\deg\gcd(\bar g,\bar h)=\deg(R/J)$. Since $J$ is minimally generated, one cannot have $t=d_2$. Indeed, if $t=d_2$, then $\bar g$ divides $\bar h$, and hence
$h\in(x_1,g)$, contradicting minimality. Thus
\[
0\le\deg(R/J)\le d_2-1.
\]

Conversely, let $0\le t\le d_2-1$. Identifying
$S\simeq k[x_2,\ldots,x_n]$, choose homogeneous forms $\ell,u,v\in S$ such
that $\deg\ell=t$, $\deg u=d_2-t$ and $\deg v=d_3-t$ with $\gcd(u,v)=1$ and $v\notin(u)$. Such forms exist because $n\ge3$. Set $J=(x_1,\ell u,\ell v)\subseteq R$. Then $J$ is minimally generated, has degree vector $\mathbf d$, and satisfies
$\gcd(x_1,\ell u,\ell v)=1$.
Moreover $\deg\gcd(\ell u,\ell v)=\deg\ell=t$. Hence Proposition~\ref{prop:numerical-bounds-non-equigenerated}{\rm(iv)}
gives
\[
\indeg(\syz(J))=d_2+1
\qquad\text{and}\qquad
\deg(R/J)=t.
\]
Thus the pair $(d_2+1,t)$ is realizable for every
$0\le t\le d_2-1$, while the value $t=d_2$ is not realizable.
\end{proof}

We next record two realizable pairs on the top boundary $e=d_1+d$ when the
two largest generator degrees coincide.
\begin{Proposition}\label{prop:top-boundary-structural-spectrum}
Assume $\mathbf d=(d_1,d,d)$ with $1\le d_1\le d$ and $n\ge3$. Then
$I_{d_1+d,\mathbf d}=[0,d_1d]\cap\mathbb Z$. Moreover:
\begin{enumerate}
\item[\rm(i)] An ideal $J$ with degree vector $\mathbf d$ realizes the
admissible pair $(d_1+d,0)$ if and only if $J$ is a complete intersection.
\item[\rm(ii)] If $d_1<d$, then the admissible pair $(d_1+d,1)$ is realizable.
\end{enumerate}
\end{Proposition}
\begin{proof}
Let $J$ have degrees $(d_1,d,d)$ and set $e=\indeg(\syz(J))$. By
Proposition~\ref{prop:numerical-bounds-non-equigenerated}{\rm(i)},
$d<e\le d_1+d$, and
\[
F_{\mathbf d}(e)-d_1d=(e-d)(e-d_1-d).
\]
Hence $F_{\mathbf d}(e)\le d_1d$, with equality if and only if
$e=d_1+d$. Moreover, $F_{\mathbf d}(d_1+d)=d_1d$ and
$\sigma-(d_1+d)D=-(d_1^2+d_1d+d^2)<0$, so
$I_{d_1+d,\mathbf d}=[0,d_1d]\cap\mathbb Z$.

For {\rm(i)}, if $J$ realizes $(d_1+d,0)$, then $\deg(R/J)=0$, and hence
$J$ is a complete intersection. Conversely, if $J$ is a complete
intersection, then $\deg(R/J)=0$ and
Proposition~\ref{prop:not-equigen-Bour-consequences}{\rm(iv)} gives
$\Bour(J)=d_1d$. Thus $F_{\mathbf d}(e)=d_1d$, and the factorization above
forces $e=d_1+d$. Therefore $J$ realizes $(d_1+d,0)$.

For {\rm(ii)}, if $d_1=1$, the assertion follows from
Proposition~\ref{prop:linear-generator-structural-spectrum}; if
$2\le d_1<d$, it follows from
Example~\ref{ex:next-extremal-non-equigenerated}.
\end{proof}
The preceding results show that numerical admissibility does not guarantee
realizability. Theorem~\ref{thm:every-e-occurs-structurally} proves that every
numerically permitted initial syzygy degree occurs, whereas
Remark~\ref{rem:universal-structural-gap} exhibits a nonrealizable admissible
pair. Proposition~\ref{prop:linear-generator-structural-spectrum} completely
determines the realizable pairs when one generator is linear, while
Proposition~\ref{prop:top-boundary-structural-spectrum} gives only partial
information on the top boundary when $d_2=d_3$. This leads naturally to the
following question.

\begin{Question}\label{ques:structural-spectrum}
For a fixed degree vector $\mathbf d$, which admissible pairs $(e,j)$ occur as
\[
(e,j)=\bigl(\indeg(\syz(J)),\deg(R/J)\bigr)
\]
for a three-generated homogeneous ideal $J\subseteq R$ of degree vector
$\mathbf d$?
\end{Question}

\section{Equigenerated ideals and generalized du Plessis--Wall gaps}
\label{sec:equigenerated}

In this section, we study the Bourbaki degree of three-generated
$d$-equigenerated ideals. Throughout, let
$J=(f_1,f_2,f_3)\subseteq R=k[x_1,\ldots,x_n]$ be minimally generated by
homogeneous forms of degree $d$, with $\hht(J)\ge2$.

We first clarify the grading convention. In Section~\ref{sec:3-gen-general},
the syzygy module is graded as a submodule of
$\bigoplus_{i=1}^3R(-d_i)$; thus, if
$\nu=(a_1,a_2,a_3)$ has shifted degree $q$, then
$\deg(a_i)=q-d_i$ for every nonzero coordinate $a_i$. In the present
equigenerated setting, we instead regard $\syz(J)$ as a graded submodule of
$R^3$ and set
\[
e=\indeg(\syz(J))
\]
in the standard grading. Hence every nonzero coordinate of a syzygy of
degree $e$ has degree $e$, while the same syzygy has shifted degree $d+e$
in $R(-d)^3$. Throughout the remainder of the paper, $e$ always denotes this standard
syzygy degree.

Consequently, Proposition~\ref{prop:basic-numerical-bound}, which is stated
in terms of the shifted degree, gives only
$\Bour(J)\le(d+e)^2$. The sharper inequality
$\Bour(J)\le e^2$ requires a separate argument in the height-two case and
will be proved in Theorem~\ref{BourBoundEquigenerated}. When
$\hht(J)=3$, the ideal $J$ is a complete intersection and the inequality
is immediate.

Theorem~\ref{Bour3Gens} also specializes neatly in this setting. Since
$D=3d$, $\sigma=3d^2$, and the shifted syzygy degree is $d+e$, it gives
\[
\Bour(J)+\deg(R/J)
=3d^2-(d+e)(2d-e)
=d^2-ed+e^2.
\]
Therefore
\begin{equation}\label{bourequigeneratedformula}
\Bour(J)=d^2-ed+e^2-\deg(R/J).
\end{equation}
This is the same numerical expression that occurs for gradient ideals of
projective plane curves, so the Bourbaki degree naturally extends the
notion introduced in~\cite{MAA}.

\subsection{A distinguished upper bound for the Bourbaki degree}
The following result generalizes \cite[Theorem~2.10]{MAA} and yields bounds
for $\deg(R/J)$ analogous to the du Plessis--Wall bounds for the global
Tjurina number of plane curves~\cite[Theorem 3.2]{CTC-Plessis}. Under the additional
hypothesis that $\syz(J)$ is locally free on the punctured spectrum, the
inequality $\Bour(J)\le e^2$ was proved in \cite[Theorem~3.3]{MFA}; this
hypothesis is automatic when $n=3$. Thus, in the case corresponding to
$d_1=0$ in the notation of \cite{MFA}, the following theorem removes the
local-freeness hypothesis and gives an affirmative answer to
\cite[Question~3.4]{MFA} within this subclass.

\begin{Theorem}\label{BourBoundEquigenerated}
Let $J=(f_1,f_2,f_3)\subseteq R=k[x_1,\ldots,x_n]$ be a $d$-equigenerated ideal with $n\geq 3$ and $\hht(J)=2$. If $e = \indeg(\syz(J))$, then
\[
\Bour(J)\leq e^2.
\]
Consequently,
\[
d(d-e)\leq \deg(R/J)\leq d^2-de+e^2.
\]
\end{Theorem}
\begin{proof}
After replacing the generators by general $k$-linear combinations, we may assume that $f_1,f_2$ is an $R$-regular sequence, since $R$ is Cohen--Macaulay and $\hht(J) = 2$. Let $C=(f_1,f_2):f_3$. Multiplication by $f_3$ gives the exact sequence
\[
0\to R/C(-d)\to R/(f_1,f_2)\to R/J\to 0.
\]
Hence
\begin{equation}\label{formula1}
    \deg(R/J)=d^2-\deg(R/C).
\end{equation}
We claim that $\deg(R/C)\leq ed$. If $e=d$, this is clear, since $\deg(R/C)\leq \deg(R/(f_1,f_2))=d^2=ed$.
Assume now that $e<d$. Let
\[
\mathfrak z=\begin{bmatrix}p\\ q\\ r\end{bmatrix}
\]
be a nonzero minimal syzygy of $J$ of degree $e$. We first note that $r\neq 0$. Indeed, if $r=0$, then $pf_1+qf_2=0$. Since $f_1,f_2$ is a regular sequence, we would have $(p,q)=a(-f_2,f_1)$ for some homogeneous $a\in R$ of degree $e-d<0$, which is impossible.
Since $\mathfrak z$ is a syzygy, it follows that $r\in C$. Moreover, $\deg r=e$.

Choose a general $k$-linear combination
\[
g=\lambda f_1+\mu f_2.
\]
We claim that $g$ may be chosen so that $\hht(r,g)=2$. Indeed, since
$f_1,f_2$ is a regular sequence, they have no common irreducible factor.
Let $h_1,\ldots,h_s$ be the distinct irreducible factors of $r$. For each
$i$, the condition
\[
h_i\mid \lambda f_1+\mu f_2
\]
defines a proper linear condition on $(\lambda,\mu)\in k^2$, because $h_i$ does not divide both $f_1$ and $f_2$. Indeed, the condition can be given as $\lambda \overline{f}_1 + \mu \overline{f_2} = 0$ in the quotient ring $R/(h_i)$, and thus the locus of such $(\lambda, \mu) \in k^2$ is the kernel of the $k$-linear map $k^2 \xrightarrow{(f_1, f_2)} R/(h_i)$. Moreover, it gives a proper subspace since each $h_i$ cannot be a factor of both $ f_1$ and $ f_2$, so either $(1,0)$ or $(0,1)$ is not in the kernel. Since $k$ is infinite, we may choose $(\lambda,\mu)$ outside the union of these finitely many proper linear
subspaces. Then no irreducible factor of $r$ divides $g$, so $\gcd(r,g)=1$. Hence $\hht(r,g)=2$.

Since $g\in (f_1,f_2)\subseteq C$, we get 
$(r,g)\subseteq C$ where $(r,g)$ is a complete  intersection of type $(e,d)$. Therefore
\[
\deg(R/C)\leq \deg(R/(r,g))=ed.
\]
Consequently,
\[
\deg(R/J)=d^2-\deg(R/C)\geq d^2-ed=d(d-e).
\]
Finally, using the Bourbaki degree formula~\eqref{bourequigeneratedformula}
we obtain
\[
\Bour(J)
=d^2+e^2-ed-\deg(R/J)
\leq d^2+e^2-ed-d(d-e)
=e^2.
\]
This proves the assertion.
\end{proof}

The following family shows that the upper bound in
Theorem~\ref{BourBoundEquigenerated} is sharp for every $1\le e<d$,
even among gradient ideals of plane curves.

\begin{Example}\label{Bour-e-square-gradient}
Let $d\geq 2$ and $1\leq e<d$. Put $a=d-e$ and
\[
F=x^{d+1}+y^{d+1}+x^{a+1}z^e,
\]
a reduced singular curve. Then the gradient ideal $J_F=(F_x,F_y,F_z)$ is generated in degree $d$ and, up to nonzero scalars,
\[
J_F=(x^aP,\ y^d,\ x^aQ),
\qquad
P=(d+1)x^e+(a+1)z^e,\quad Q=exz^{e-1}.
\]
Since $\gcd(P,Q)=1$, the sequence $P,y^d,Q$ is regular. We claim that
$\syz(J_F)$ is minimally generated by
\[
\eta_1=\begin{bmatrix}Q\\0\\-P\end{bmatrix},\qquad
\eta_2=\begin{bmatrix}y^d\\-x^aP\\0\end{bmatrix},\qquad
\eta_3=\begin{bmatrix}0\\-x^aQ\\y^d\end{bmatrix}.
\]
Indeed, if $(u,v,w)\in \syz(J_F)$, then $u(x^aP)+vy^d+w(x^aQ)=0$ hence $x^a(uP+wQ)=-vy^d$. Since $\gcd(x^a,y^d)=1$, we have $v=x^av'$. Thus 
$uP+v'y^d+wQ=0$. As $P,y^d,Q$ is a regular sequence, $(u,v',w)$ is generated by the Koszul
syzygies of $P,y^d,Q$, and therefore $(u,v,w)$ is generated by
$\eta_1,\eta_2,\eta_3$.

The generators are minimal: their coefficient degrees are $e,d,d$, respectively,
and $\eta_2,\eta_3$ cannot be redundant modulo $\eta_1$, since this would force
$P$ and $Q$ to be scalar multiples. Hence
\[
\indeg(\syz(J_F))=e.
\]
Moreover, $y^d\eta_1-Q\eta_2+P\eta_3=0$ and this gives the minimal free resolution
\[
0\to R(-2d-e)
\to R(-d-e)\oplus R(-2d)^2
\to R(-d)^3
\to R
\to R/J_F
\to 0.
\]
Thus $\deg(R/J_F)=d(d-e)$. By the Bourbaki degree formula,
\[
\Bour(J_F)=d^2+e^2-de-\deg(R/J_F)
=d^2+e^2-de-d(d-e)=e^2.
\]
Therefore, for every $1\leq e<d$, there exists a plane curve whose gradient ideal
satisfies $\indeg(\syz(J_F))=e$ and $\Bour(J_F)=e^2$.
\end{Example}

The equality case in Theorem~\ref{BourBoundEquigenerated} admits the
following homological characterization.

\begin{Proposition}\label{prop:BourJ-max-res}
Let $J=(f_1,f_2,f_3)\subseteq R=k[x_1,\ldots,x_n]$ be a $d$-equigenerated ideal with $n\geq 3$ and $\hht(J)=2$.  Then the following are equivalent:
\begin{enumerate}
    \item[(a)] $\Bour(J)=e^2$
    \item[(b)] The minimal graded free resolution of $R/J$ is
    \[0\to R(-2d-e)\to R(-d-e)\oplus R(-2d)^2\to R(-d)^3\to R \to R/J \to 0\]
\end{enumerate}
\end{Proposition}
\begin{proof} 
Assume first that $\Bour(J)=e^2$, and note that this forces $e < d$, since $\hht(J) = 2$ gives $\deg(R/J) > 0$. After a general change of generators, we may assume that $f_1,f_2$ is a regular sequence. Set $C=(f_1,f_2):f_3$. As in the proof of Theorem~\ref{BourBoundEquigenerated}, there exists a regular sequence $r,g\in C$, with
$\deg r=e$ and $\deg g=d$.

By the Bourbaki degree formula and~\eqref{formula1}, this is equivalent to
\[
\deg(R/C)=ed=\deg(R/(r,g)).
\]
Since $(r,g)\subseteq C$, we have an exact sequence
\[
0\to C/(r,g)\to R/(r,g)\to R/C\to 0.
\]
If $C/(r,g)\neq 0$, then, because $R/(r,g)$ is Cohen--Macaulay of dimension $n-2$, every associated prime of the submodule $C/(r,g)$ has dimension $n-2$. Hence $\dim C/(r,g)=n-2$. By additivity of degree in the above exact sequence, this would give
\[
\deg(R/(r,g))=\deg(C/(r,g))+\deg(R/C)>\deg(R/C),
\]
contradicting $\deg(R/(r,g))=\deg(R/C)$. Therefore $C/(r,g)=0$, and so
$C=(r,g)$.

Now use the standard exact sequence
\[
0\to R/C(-d)\xrightarrow{\cdot f_3} R/(f_1,f_2)\to R/J\to 0.
\]
Since $C=(r,g)$ is a complete intersection of type $(e,d)$, the minimal
resolution of $R/C(-d)$ is
\[
0\to R(-2d-e)\to R(-d-e)\oplus R(-2d)\to R(-d)\to R/C(-d)\to 0.
\]
Together with the Koszul resolution of $R/(f_1,f_2)$, the mapping cone gives
\[
0\to R(-2d-e)\to R(-d-e)\oplus R(-2d)^2
\to R(-d)^3\to R\to R/J\to 0.
\]
No cancellation occurs, so this resolution is minimal. Thus (b) holds.

Conversely, assume that $R/J$ has the displayed minimal graded free resolution. Then
\[
\Hs_{R/J}(t)=
\frac{1-3t^d+t^{d+e}+2t^{2d}-t^{2d+e}}{(1-t)^n}.
\]
Set $h(t)=1-3t^d+t^{d+e}+2t^{2d}-t^{2d+e}$. Since $\hht(J)=2$, the numerator $h(t)$ has a zero of order $2$ at $t=1$. Hence
\[
\deg(R/J)=\frac{h''(1)}{2}=\frac{2(d^2-ed)}{2}=d^2-ed.
\]
Using the Bourbaki degree formula, we obtain
\[
\Bour(J)=d^2+e^2-ed-(d^2-ed)=e^2.
\]
Thus $\Bour(J)=e^2$, as required.
\end{proof}

\subsection{Gaps and du Plessis--Wall bounds}\label{num-gaps}

Let $J=(f_1,f_2,f_3)\subseteq R$ be minimally generated by three
homogeneous forms of degree $d\ge2$, with $\hht(J)=2$; in other words,
$J$ is a $d$-equigenerated strict almost complete intersection. Set
$e=\indeg(\syz(J))$. By the Bourbaki degree formula and $0\le\Bour(J)\le e^2$, we obtain
\[
d(d-e)\le\deg(R/J)\le d^2-ed+e^2.
\]
By analogy with \cite{CTC-Plessis}, we call these inequalities the
generalized du Plessis--Wall bounds. For each fixed $e$, define the
integer intervals
\[
I_{d,e}:=\bigl[d(d-e),\,d(d-e)+e^2\bigr]\cap\mathbb Z,
\qquad
B_{d,e}:=[0,e^2]\cap\mathbb Z.
\]
Thus $I_{d,e}$ is the numerically allowed interval for $\deg(R/J)$,
whereas $B_{d,e}$ is the numerically allowed interval for $\Bour(J)$.

The interesting range is $e<d$. Indeed, when $e=d$, the preceding
inequalities reduce to
\[
0\le\deg(R/J)\le d^2,
\qquad
0\le\Bour(J)\le d^2,
\]
so these estimates alone detect no numerical gap.
For comparison with the classical du Plessis--Wall picture, we also
include the formal block $I_{d,0}=\{d^2\}$, although $e=0$ does not
occur for a strict almost complete intersection.
\medskip

\noindent
The following table records the intervals for $d=2,3,4,5$.

\begin{table}[h]
\centering
\renewcommand{\arraystretch}{1.15}
\begin{tabular}{|c|c|c|c|c|c|c|c|}
$d$ & $e=0$ & $e=1$ & $e=2$ & $e=3$ & $e=4$
& Union for $0\le e<d$ & Gaps \\ \hline
$2$ & $\{4\}$ & $[2,3]$ & --- & --- & ---
& $\{2,3,4\}$ & none \\
$3$ & $\{9\}$ & $[6,7]$ & $[3,7]$ & --- & ---
& $\{3,4,5,6,7,9\}$ & $8$ \\
$4$ & $\{16\}$ & $[12,13]$ & $[8,12]$ & $[4,13]$ & ---
& $\{4,5,\ldots,13,16\}$ & $14,15$ \\
$5$ & $\{25\}$ & $[20,21]$ & $[15,19]$ & $[10,19]$ & $[5,21]$
& $\{5,6,\ldots,21,25\}$ & $22,23,24$
\end{tabular}
\end{table}

For $1\le e<d$, one has
\[
I_{d,e}\subseteq I_{d,d-1}=[d,d^2-d+1]\cap\mathbb Z.
\]
Indeed, $d(d-e)\ge d$, while
\[
d^2-ed+e^2\le d^2-d+1
\]
because $e(d-e)\ge d-1$. Consequently,
\[
\bigcup_{1\le e<d}I_{d,e}
=I_{d,d-1}
=[d,d^2-d+1]\cap\mathbb Z.
\]
After adjoining the formal block $I_{d,0}=\{d^2\}$, the integers missing
from $[d,d^2]\cap\mathbb Z$ are
\[
G_d:=\{d^2-b\mid 1\le b\le d-2\}.
\]
Thus $G_2=\emptyset$, while, for $d\ge3$,
\[
G_d=\{d^2-d+2,\ldots,d^2-1\}.
\]

For completeness, the number of integers separating two consecutive
blocks $I_{d,e}$ and $I_{d,e-1}$ is
\[
\min I_{d,e-1}-\max I_{d,e}-1=d-e^2-1,
\]
whenever this number is positive. Thus these two blocks are separated
precisely when $d>e^2+1$. For $e\ge2$, however, any such local gap is contained in the numerically allowed block $I_{d,d-1}$. The only gap that survives in
the total union is the one between the formal block $e=0$ and the block
$e=1$, namely $G_d$, which has cardinality $d-2$.

\medskip

The behavior of the Bourbaki degree is different. The intervals
\[
B_{d,e}=[0,e^2]\cap\mathbb Z
\]
are nested:
\[
B_{d,1}\subseteq B_{d,2}\subseteq\cdots\subseteq B_{d,d-1}.
\]
Consequently,
\[
\bigcup_{1\le e<d}B_{d,e}
=[0,(d-1)^2]\cap\mathbb Z,
\]
so the Bourbaki degree has no numerical gaps within its numerically
allowed range for $e<d$.

The relation between the two invariants is governed by the Bourbaki
formula. For each fixed $e$,
\[
\deg(R/J)=d^2-ed+e^2-\Bour(J).
\]
Thus, at the numerical level, the two intervals are related by an
order-reversing correspondence. The intervals for $\deg(R/J)$ move as $e$ varies
because their left endpoints are $d(d-e)$, whereas all the intervals
for $\Bour(J)$ begin at $0$.

On the boundary block $e=d$, the Bourbaki formula becomes
\[
\Bour(J)+\deg(R/J)=d^2.
\]
Suppose that $\deg(R/J)=j\in G_d$ for some ideal $J$. Since $j$ does not belong to
any block with $1\le e<d$ and one always has $e\le d$, necessarily
$e=d$. Writing $j=d^2-b$, one then has
\[
\Bour(J)=b,\qquad 1\le b\le d-2.
\]
Hence the gap value $d^2-b$ is filled precisely when the boundary pair
\[
\bigl(e,\deg(R/J)\bigr)=(d,d^2-b),
\]
or equivalently
\[
\bigl(e,\Bour(J)\bigr)=(d,b),
\]
is realizable. The numerical bounds alone do not exclude these boundary
pairs. The next proposition shows that the case $b=1$ cannot occur.

\begin{Proposition}\label{prop:nearly-free-restriction}
Let $R = k[x_1, \ldots, x_n]$, $J = (f_1, f_2, f_3) \subseteq R$ a $d$-equigenerated ideal with $d > 1$, $\hht(J) = 2$. If $e = \indeg(\syz(J))$ and $\deg(R/J) = d(d-e)+e^2-1$, then $2e \leq d+1$.
\end{Proposition}

\begin{proof}
The hypothesis and the Bourbaki degree formula give $\Bour(J)=1$.
Consequently, $\syz(J)$ has a minimal graded free resolution of
nearly-free type~\cite[Theorem~2.4 and Proposition~2.9]{MFA}:
\[
0\to R(e-d-2)
\to R(e-d-1)^2\oplus R(-e)
\to\syz(J)\to0.
\]
Since $e$ is the initial degree, one has
\[
e\le d-e+1,
\]
and hence $2e\le d+1$.
\end{proof}

A similar conclusion arises in the \emph{free} case, that is, when
$\Bour(J)=0$. For $e=\indeg(\syz(J))$, one has
$\deg(R/J)=d(d-e)+e^2$. Moreover, in the standard grading,
\[
\syz(J)\simeq R(-e)\oplus R(-(d-e)),
\]
so the minimality of $e$ implies $e\le\lfloor d/2\rfloor$. The following
lemma gives another restriction, this time for high Bourbaki degree.

\begin{Lemma}\label{numerical-reduction}
If $\Bour(J)=(d-1)^2+1$, then $e=d$ and $\deg(R/J)=2d-2$.
\end{Lemma}

\begin{proof}
Since $J$ is generated by three forms of degree $d$, the Koszul
relations give $e\le d$, while the general bound $\Bour(J)\le e^2$
gives $(d-1)^2+1\le e^2$. Hence $e=d$, and the Bourbaki degree formula
yields $\deg(R/J)=2d-2$.
\end{proof}

This discussion and these formulas are independent of the dimension
$n=\dim R$. When $n=3$, the ideal $J$ may be of \emph{gradient type},
that is,
\[
J=J_F=(F_x,F_y,F_z)
\]
for some reduced homogeneous polynomial $F\in R=k[x,y,z]$ of degree
$d+1$. In this case, $\deg(R/J_F)=\tau(F)$ is the global Tjurina number
of the reduced plane curve $V(F)\subseteq\PP^2$. The preceding
discussion therefore specializes directly to the projective plane curve
setting, as in the classical work of \cite[Section~4]{CTC-Plessis}.

Motivated by these different possible behaviors, we introduce the
following notions.

\begin{Definition}\label{defi:gaps}
Let $n=\dim R\ge3$ and fix $d\ge2$. A pair $(e,j)$ with $1\le e\le d$
is \emph{numerically admissible} if $j\in I_{d,e}$.
\begin{itemize}
\item[(a)] A numerically admissible pair $(e,j)$ that cannot be realized
by any $d$-equigenerated ideal $J=(f_1,f_2,f_3)$ of height $2$ in
dimension $n$ is called a \emph{structural gap} in degree $d$ and
dimension $n$. Otherwise, the pair is called \emph{realizable}.

\item[(b)] A pair that is a structural gap in some dimension $n_0$ but
is realizable in a higher dimension is called a \emph{dimensional gap}.

\item[(c)] When $n=3$, a realizable pair $(e,\deg(R/J))$ is an
\emph{integrable gap} in degree $d$ if there is no reduced homogeneous
polynomial $F\in R_{d+1}$ such that $J_F$ realizes the pair.
\end{itemize}
\end{Definition}

Proposition~\ref{prop:nearly-free-restriction} gives the structural gap
$(d,d^2-1)$ in every dimension $n\ge3$ and every degree $d\ge2$. For
$d=2$, the pair $(2,2)$ is a dimensional gap: it cannot be realized in
dimension $3$ by Theorem~\ref{low-degree-gap}, but it is realized in
dimension $4$ by
\[
J=(x_1x_4,x_2x_3,x_1x_3-x_2x_4).
\]
Consequently, it is realizable in every dimension $n\ge4$. More
generally, if a pair $(e,j)$ is realized by an ideal
$J\subseteq k[x_1,\ldots,x_n]$, then it is realizable in every dimension
$m\ge n$ by considering the flat extension
$\widetilde J=Jk[x_1,\ldots,x_m]$.

The integrable gaps described in Definition~\ref{defi:gaps}{\rm(c)}
mimic the phenomenon observed for $2\times4$ matrices in
\cite[Section~4]{MFA}, where the gap $\Bour(A)\neq2$ for Jacobian
matrices is filled by a non-Jacobian matrix. We may ask the following.

\begin{Question}\label{question:integrable-gaps}
When $n=3$ and $J=(f_1,f_2,f_3)\subseteq R$ is a $d$-equigenerated
ideal, are there realizable pairs $(e,\deg(R/J))$ which are integrable
gaps?
\end{Question}

For $d\leq 3$ the answer is negative: all realizable pairs are
already achieved by gradient ideals, as will be shown in
Subsection~\ref{subsec:ideals-quadrics} for $d = 2$ and in Section~\ref{sec:quartic-curves} for $d = 3$. In fact, for $d = 2, 3$, we
prove the stronger statement that every possible minimal graded free resolution in dimension $3$ can be realized by a gradient ideal for these degrees.

In the free and nearly free cases, the only restrictions are
$e\le\lfloor d/2\rfloor$ and $e\le\lfloor(d+1)/2\rfloor$, respectively.
When $n=3$, every corresponding free or nearly free pair $(e,j)$ is
realizable by a gradient ideal; see
\cite[Theorems~1.1, 1.2, and 3.1]{dimca2017exponents}.

\section{d--equigenerated ideals of homological dimension 2}\label{sec:hom-dim-2}

Let $J=(f_1,f_2,f_3)\subseteq R$ be a $d$-equigenerated strict almost
complete intersection of height $2$, with $d\ge2$. Throughout this
section, we assume that $J$ has projective dimension $2$; for example,
this always holds when $n=3$ and $\Bour(J)\neq0$. According to
\cite{BFRS}, the minimal graded free resolution of $R/J$ has the form
\begin{equation}\label{res-prelim-bis}
0\to
\bigoplus_{j=1}^{m-2}R(-d-\delta_{j+2}-\epsilon_j)
\to
\bigoplus_{i=1}^{m}R(-d-\delta_i)
\to R(-d)^3\to R\to R/J\to0,
\end{equation}
for certain integers $m\ge3$,
$1\le\delta_1\le\delta_2\le\delta_3\le\cdots\le\delta_m$ and positive integers $\epsilon_1,\ldots,\epsilon_{m-2}$ satisfying
\begin{equation}\label{Dnondecresing}
\delta_3+\epsilon_1\le\cdots\le\delta_m+\epsilon_{m-2},
\end{equation}
\begin{equation}\label{sumHS-c1}
\delta_1+\delta_2
=d+\sum_{j=1}^{m-2}\epsilon_j,
\end{equation}
\begin{equation}\label{c2}
\delta_i+\delta_j\ge d+1
\qquad\text{for every }1\le i<j\le m,
\end{equation}
and
\begin{equation}\label{c3}
\delta_3\le d.
\end{equation}

\begin{Remark}\label{numberofsyzygies}
Since the $\epsilon_i$ are positive, \eqref{sumHS-c1}, the ordering of
the $\delta_i$, and \eqref{c3} give
\[
m-2
\le\sum_{i=1}^{m-2}\epsilon_i
=\delta_1+\delta_2-d
\le d.
\]
\end{Remark}

\subsection{Low-degree gaps for the Bourbaki degree}

We begin with two low-degree gaps for the Bourbaki degree of ideals of
homological dimension $2$.

\begin{Theorem}\label{low-degree-gap}
Let $R=k[x_1,\ldots,x_n]$, with $n\ge3$, and let
$J=(f_1,f_2,f_3)\subseteq R$ be a $d$-equigenerated strict almost
complete intersection of height $2$ and homological dimension $2$. Then:
\begin{enumerate}
\item[{\rm(a)}] If $d=2$, then $\Bour(J)\neq2$.
\item[{\rm(b)}] If $d=3$, then $\Bour(J)\neq5$.
\end{enumerate}
\end{Theorem}

\begin{proof}
For $\mathrm{(a)}$, assume that $d = 2$ and $\Bour(J)=2$. By Lemma~\ref{numerical-reduction}, we get $e=2$ and $\deg(R/J)=2$.  Let 
\begin{equation}\label{res-prelimJ=2}
	0\to \bigoplus_{j=1}^{m-2} R(-2-\delta_{j+2}-\epsilon_j)\to \bigoplus_{i=1}^{m} R(-2-\delta_i)\to R(-2)^3\to J \to 0,
\end{equation} be the minimal graded free resolution of $J$. Since $\delta_1=e=d$  then, from the condition \eqref{c3}, we conclude that 
\begin{equation}
\delta_1=\delta_2=\delta_3=2
\end{equation} 
Using these equalities, together with conditions \eqref{Dnondecresing}-\eqref{c3} and the Remark \ref{numberofsyzygies}, we conclude that the possible vectors $(m;\delta_1,\ldots,\delta_m;\epsilon_1,\ldots,\epsilon_{m-2})$ are:
\[ (3;2,2,2;2)\,\,\mbox{or}\,\, (4;2,2,2,\delta_4;1,1)\]
Now, using the resolution \eqref{res-prelimJ=2} to calculate the degree of $R/J$ in each of these possibilities we get:
\[\deg(R/J)=0\,\,\mbox{or}\,\,\deg(R/J)=3-\delta_4\]
On the other hand, $\deg(R/J)=2$. Thus, $3-\delta_4=2,$ that is, $\delta_4=1$. But this is a contradiction because $2=\delta_3\leq \delta_4$. 

(b) Assume that $\Bour(J)=5$. 
By Lemma~\ref{numerical-reduction}, we get $e=3$ and $\deg(R/J)=4$.
Let 
\begin{equation}\label{res-prelimJ=3}
	0\to \bigoplus_{j=1}^{m-2} R(-3-\delta_{j+2}-\epsilon_j)\to \bigoplus_{i=1}^{m} R(-3-\delta_i)\to R(-3)^3\to J \to 0,
\end{equation} be the minimal graded free resolution of $J$. Since $\delta_1=e=d$  then, from the condition \eqref{c3}, we conclude that  $\delta_1=\delta_2=\delta_3=d$. Together with conditions \eqref{Dnondecresing}-\eqref{c3} and the Remark \ref{numberofsyzygies}, we conclude that the possible vector $(m;\delta_1,\ldots,\delta_m;\epsilon_1,\ldots,\epsilon_{m-2})$ are:
\[(3;3,3,3;3),\,\, (4;3,3,3,\delta_4;1,2),\,\, (4;3,3,3,\delta_4;2,1)\,\,\mbox{or}\,\, (5;3,3,3,\delta_4,\delta_5;1,1,1)\]
Now, using the resolution \eqref{res-prelimJ=3} to calculate the degree of $R/J$ in each of these possibilities we get:
\[\deg (R/J)=0,\,\,\deg(R/J)=-2\delta_4+8,\,\,\deg(R/J)=-\delta_4+5\,\,\mbox{or}\,\,\deg(R/J)=-\delta_4-\delta_5+9.\]
On the other hand, from Lemma \ref{numerical-reduction} we have $\deg(R/J)=4$. Thus, we have one of the following possibilities:
\[\delta_4=2,\quad \delta_4=1\quad \mbox{or}\quad \delta_4+\delta_5=5\]
But this is a contradiction because $3=\delta_3\leq \delta_4\leq\delta_5$.
\end{proof}

The hypothesis on the homological dimension is essential in
Theorem~\ref{low-degree-gap}{\rm(a)}. Indeed, in
$R=k[x_1,x_2,x_3,x_4]$, the ideal $J=(x_1x_4,x_2x_3,x_1x_3-x_2x_4)$ satisfies
$\indeg(\syz(J))=2$, $\deg(R/J)=2$ and $\Bour(J)=2$. Thus the value $2$ is realizable in dimension $4$, although not by an ideal of homological dimension $2$.
The cases $d=2$ and $d=3$ suggest the tempting expectation that
\[
\Bour(J)\neq(d-1)^2+1
\qquad\text{for every }d\ge2.
\]
The following family shows that this fails for every $d\ge4$.

\begin{Proposition}\label{prop:Jd-family}
	Let $R=k[x,y,z]$ and let $d\geq 4$. Consider
	\[
	J_d=\bigl(
	x^d,\ 
	y^{d-1}z-(d-2)xyz^{d-2},\ 
	y^d-(d-1)xy^{d-2}z+x^{d-2}z^2
	\bigr).
	\]
	Then $\hht(J_d)=2$,  $\deg(R/J_d)=2d-2$ and  $\indeg(\syz(J_d))=d$. Consequently,
	\[
	\Bour(J_d)=(d-1)^2+1.
	\]
\end{Proposition}
\begin{proof}
Set
\[
f_1=x^d,\qquad
f_2=y^{d-1}z-(d-2)xyz^{d-2},\qquad
f_3=y^d-(d-1)xy^{d-2}z+x^{d-2}z^2.
\]
Since $f_1=x^d$ and $f_3\notin(x)$, the generators have no common
divisor, so $\hht(J_d)\ge2$. Since $[0:0:1]\in V(J_d)$, we conclude
that $\hht(J_d)=2$.

We first compute the degree. Since $V(J_d)\cap V(z)=\emptyset$, we may
dehomogenize with respect to $z$. Thus
\[
\deg(R/J_d)=\length_k(B/I_d),
\]
where $B=k[x,y]$ and
\[
I_d=(x^d,yq,h),
\quad
q=y^{d-2}-(d-2)x,
\quad
h=y^d-(d-1)xy^{d-2}+x^{d-2}.
\]
Multiplication by $y$ gives an exact sequence
\[
0\to B/(I_d:y)\xrightarrow{\cdot y}B/I_d\to B/(I_d,y)\to0.
\]
Now $(I_d,y)=(y,x^{d-2})$, and hence
\[
\length_k(B/(I_d,y))=d-2.
\]

We claim that
\[
(I_d:y)=(x^d,q,h)=:K.
\]
Modulo $q$, one has $x=y^{d-2}/(d-2)$, and a direct substitution shows
that
\[
B/K\simeq k[y]/(y^d).
\]
Indeed, the image of $h$ is $y^d$ times a polynomial coprime to $y$:
this polynomial is $-1/4$ if $d=4$ and has constant term $1$ if $d>4$.

The inclusion $K\subseteq(I_d:y)$ is clear. Conversely, if
\[
ay=Ax^d+B\,yq+Ch,
\]
then, modulo $y$, one has $Ax^2+C\in(y)$, so $C=-Ax^2+yD$ for some
$D\in B$. Therefore
\[
ay=y(AT+Bq+Dh),
\qquad
T=-x^2y^{d-1}+(d-1)x^3y^{d-3}.
\]
The image of $T$ in $B/K\simeq k[y]/(y^d)$ is zero, so $T\in K$ and
hence $a\in K$. Thus $(I_d:y)=K$, and consequently
\[
\deg(R/J_d)
=\length_k(B/I_d)
=d+(d-2)=2d-2.
\]

It remains to prove that $e=\indeg(\syz(J_d))=d$. Since the Koszul
relations have coefficient degree $d$, one has $e\le d$. Suppose that
\[
af_1+bf_2+cf_3=0
\]
with $a,b,c\in R_m$ and $m<d$. It is enough to consider $m=d-1$, since
a syzygy of smaller degree may be multiplied by a suitable power of
$z$.

Dehomogenizing with respect to $z$, write
\[
b=\sum_{i=0}^{d-1}x^ib_i(y),
\qquad
c=\sum_{i=0}^{d-1}x^ic_i(y),
\qquad
\deg b_i,\deg c_i\le d-1-i.
\]
Since $bf_2+cf_3\in(x^d)$, comparison of the coefficient of $x^r$ gives
\begin{equation}\label{eq:Jd-coefficients}
y^{d-1}(b_r+yc_r)
-(d-2)yb_{r-1}
-(d-1)y^{d-2}c_{r-1}
+c_{r-d+2}=0,
\end{equation}
where $b_j=c_j=0$ outside the range $0\le j\le d-1$.

Put $N=d-2$, $P_i=b_i+yc_i$, and
\[
U(y)=N-(N+1)y^{N-2}.
\]
Then \eqref{eq:Jd-coefficients} becomes
\begin{equation}\label{eq:Jd-recursion-expanded}
y^{N+1}P_i-NyP_{i-1}+y^2U(y)c_{i-1}+c_{i-N}=0
\qquad(0\le i\le N+1),
\end{equation}
with out-of-range indices interpreted as zero and
$\deg b_i,\deg c_i\le N+1-i$.

\smallskip\noindent
\emph{The case $N\ge4$.}
The equation for $i=0$ gives $P_0=0$, hence $b_0=-yc_0$ and
$\deg c_0\le N$. For $i=1$ we get
\[
y^{N-1}P_1=-U(y)c_0.
\]
Since $U(0)=N\neq0$, there are $\alpha,\beta\in k$ such that
\[
c_0=y^{N-1}(\alpha+\beta y),\qquad
P_1=-U(y)(\alpha+\beta y).
\]
The equation for $i=2$ is therefore
\[
y^{N+1}P_2=-yU(y)\bigl(N\alpha+N\beta y+yc_1\bigr).
\]
Divisibility by $y^{N+1}$ first gives $\alpha=0$, and then
$N\beta+c_1\in(y^{N-1})$. Write
$c_1=-N\beta+y^{N-1}(\gamma+\delta y)$.
The identity $b_1=P_1-yc_1$ gives
\[
b_1=(N+1)\beta y^{N-1}-\gamma y^N-\delta y^{N+1}.
\]
Since $\deg b_1\le N$, we have $\delta=0$. Consequently,
\[
\begin{gathered}
c_0=\beta y^N,\qquad b_0=-\beta y^{N+1},\\
c_1=-N\beta+\gamma y^{N-1},\qquad
b_1=(N+1)\beta y^{N-1}-\gamma y^N,\qquad
P_2=-\gamma U(y).
\end{gathered}
\]
For $i=3$, the coefficient of $y$ in
\eqref{eq:Jd-recursion-expanded} is $N^2\gamma$, so $\gamma=0$
and $P_2=0$.

We now use induction. If $2\le i\le N-2$ and $P_i=0$, then the
equation with index $i+1$ gives
\[
y^{N-1}P_{i+1}=-U(y)c_i.
\]
Thus $y^{N-1}\mid c_i$. But $b_i=-yc_i$ and the degree bound on
$b_i$ give $\deg c_i\le N-i<N-1$. Hence $c_i=b_i=0$ and
$P_{i+1}=0$. We conclude that
\[
b_i=c_i=0\quad(2\le i\le N-2),\qquad P_{N-1}=0.
\]
The latter identity and the degree bounds imply
$c_{N-1}=u+vy$ and $b_{N-1}=-uy-vy^2$ for some $u,v\in k$.
The equation with index $N$ reads
\[
y^{N+1}P_N+y^2U(y)(u+vy)+\beta y^N=0.
\]
Since $N\ge4$, its coefficients of $y^2$ and $y^3$ give
$u=v=0$. Its coefficient of $y^N$ then gives $\beta=0$, and
therefore $P_N=0$.
Now $c_N=w$ and $b_N=-wy$ for some $w\in k$. The final equation,
with index $N+1$, becomes
\[
y^{N+1}P_{N+1}+wy^2U(y)=0.
\]
Its coefficient of $y^2$ gives $w=0$, and hence $P_{N+1}=0$.
As $b_{N+1}$ and $c_{N+1}$ are constants, they also vanish.
Thus all the $b_i,c_i$ are zero.

\smallskip\noindent
\emph{The case $N=3$, or $d=5$.}
Here $U(y)=3-4y$. The same calculations with indices $0,1,2$ give
\[
\begin{gathered}
c_0=\beta y^3,\quad b_0=-\beta y^4,\\
c_1=-3\beta+\gamma y^2,\quad
b_1=4\beta y^2-\gamma y^3,\quad
P_2=-\gamma(3-4y).
\end{gathered}
\]
The degree bound on $b_2=P_2-yc_2$ implies $c_2=u+vy$.
The equation with index $3$ is
\[
9\gamma y+(-12\gamma+3u)y^2
+(\beta+3v-4u)y^3+y^4(P_3-4v)=0.
\]
It follows that $\gamma=u=0$, $\beta=-3v$, and $P_3=4v$.
Consequently $c_3=w$ and $b_3=4v-wy$ for some $w\in k$.
The equation with index $4$ is now
\[
9v-12vy+3wy^2-4wy^3+y^4P_4=0.
\]
Thus $v=w=0$ and $P_4=0$. All the coefficients vanish.

\smallskip\noindent
\emph{The case $N=2$, or $d=4$.}
Here $U(y)=-1$. The equations with indices $0,1,2$ and the degree
bounds give
\[
\begin{gathered}
c_0=\beta y^2,\qquad b_0=-\beta y^3,\\
c_1=-\beta+\gamma y,\qquad b_1=2\beta y-\gamma y^2,\\
c_2=\eta,\qquad b_2=\gamma-\eta y.
\end{gathered}
\]
Indeed, the first two equations give $P_0=0$ and $c_0=yP_1$,
so $P_1=\alpha+\beta y$; the equation with index $2$ forces
$\alpha=0$, and the bound on $b_1$ forces $P_2$ to be a constant
$\gamma$. The displayed expressions follow.
Finally, the equation with index $3$ reads
\[
y^3P_3-\beta-\gamma y-\eta y^2=0.
\]
Hence $\beta=\gamma=\eta=0$ and $P_3=0$, so all the coefficients
vanish in this case as well.

Therefore no nonzero
syzygy has coefficient degree smaller than $d$, and consequently
\[
\indeg(\syz(J_d))=d.
\]
Finally, the Bourbaki degree formula gives
\[
\Bour(J_d)
=d^2-\deg(R/J_d)
=d^2-(2d-2)
=(d-1)^2+1.
\]
\end{proof}

\begin{Example}\label{ex:quintic-cuspidal}
Consider the quintic plane curve defined by
\[
f=x^3y^2+y^3z^2+z^3x^2.
\]
Its Jacobian ideal $J_F$ is generated by three quartics. The curve has three cusp singularities at coordinate points $p_1=[1:0:0],\, p_2=[0:1:0],\, p_3=[0:0:1]$, so
\[
\tau(f)=6=\deg(R/J_F).
\]
Moreover, a direct computation shows that $e=\indeg(\syz(J_F))=4$. 
Therefore
\[
\Bour(J_F)=4^2-6=10=(4-1)^2+1.
\]
Thus the Jacobian ideal of this quintic provides another counterexample with quartic generators.
\end{Example}

The preceding discussion leads naturally to the following problem.

\begin{Question}
For each integer $d\ge4$, determine the set of possible values of
$\Bour(J)$, where $J\subseteq k[x,y,z]$ runs through all
$d$-equigenerated strict almost complete intersections of height $2$.
Equivalently, determine the gaps among the numerically admissible
Bourbaki degrees of such ideals.
\end{Question}

\begin{Remark}
The argument used in Lemma~\ref{numerical-reduction} shows that any value
\[
\Bour(J)>(d-1)^2
\]
forces $e=d$. Indeed, since $\Bour(J)\le e^2$ and $e\le d$, the inequality
$\Bour(J)>(d-1)^2$ implies $e>d-1$, and hence $e=d$. The Bourbaki degree
formula then gives
\[
\deg(R/J)=d^2-\Bour(J).
\]
Thus the problem of understanding the upper range of possible Bourbaki
degrees is equivalent to determining which values of $\deg(R/J)$ can
occur in the extremal case $e=d$.
\end{Remark}

\subsection{Ideals with maximal degree in codimension two}
\label{subsec:ideals-max-degree}

When $J=(f_1,f_2,f_3)$ is a $d$-equigenerated strict almost complete intersection of height $2$ and homological dimension $2$, we prove the sharper bound $\deg(R/J)\le d^2-d$, which improves the general bound $\deg(R/J)\le d^2$. We also describe the possible minimal free resolutions in the extremal case $\deg(R/J)=d^2-d$.

\begin{Lemma}\label{upperboundinitialdegree}
Let $R=k[x_1,\ldots,x_n]$, with $n\ge3$, and let $J\subseteq R$ be a $d$-equigenerated strict almost complete intersection of height $2$ and homological dimension $2$. If $\deg(R/J)=d^2-d+\alpha$ for some $0\le\alpha\le d-1$, then $e=\indeg(\syz(J))\le d-1$.
\end{Lemma}

\begin{proof}
After a general change of generators, assume that $f_1,f_2$ form a
regular sequence and set $C=(f_1,f_2):f_3$. Since $f_3\notin(f_1,f_2)$,
$C$ is proper. The exact sequence
\begin{equation}\label{seqofthemappingcone}
0\to (R/C)(-d)\xrightarrow{\cdot f_3}R/(f_1,f_2)
\to R/J\to0
\end{equation}
identifies $(R/C)(-d)$ with a nonzero submodule of the
Cohen--Macaulay module $R/(f_1,f_2)$. Thus $C$ has height two, and
additivity of multiplicities gives
\[
\deg(R/C)=d^2-\deg(R/J)=d-\alpha.
\]
Moreover, $\pd_R(R/J)=3$ and $\pd_R(R/(f_1,f_2))=2$, so the
exact sequence gives $\pd_R(R/C)\le2$. Since $C$ has height two,
$R/C$ is Cohen--Macaulay.

Suppose that $e=d$. Every nonzero homogeneous element $h\in C$ of
degree $q$ yields a syzygy of coefficient degree $q$ by writing
$hf_3=af_1+bf_2$. Therefore $\indeg C\ge d$; as
$(f_1,f_2)\subseteq C$, equality holds.

Choose $n-2$ independent linear forms which form an $R/C$-regular
sequence, and let $A$ be the resulting Artinian reduction. Then
$A\simeq k[u,v]/\overline C$, where $\overline C$ has no nonzero
component of degree below $d$. Hence
\[
\deg(R/C)=\length_k A
\ge\sum_{q=0}^{d-1}\dim_k k[u,v]_q
=\sum_{q=0}^{d-1}(q+1)
=\binom{d+1}{2}>d.
\]
This contradicts $\deg(R/C)=d-\alpha\le d$. Since $e\le d$,
it follows that $e\le d-1$.
\end{proof}

\begin{Theorem} \label{structural-gap-degree-general}
Let $R=k[x_1,\ldots,x_n]$, with $n\geq 3$, and let
$J\subseteq R$ be a $d$-equigenerated strict almost complete intersection of height $2$ and homological dimension $2$. Then
\[
\deg(R/J)\leq d^2-d.
\]
In particular,
\[
\deg(R/J)\neq d^2-d+1.
\]
\end{Theorem}
\begin{proof}
If $\deg(R/J)<d^2-d$, there is nothing to prove. Assume therefore
that $\deg(R/J)\ge d^2-d$.
After a general change of generators, take $f_1,f_2$ to be a regular
sequence and set $C=(f_1,f_2):f_3$. Minimality of the generators makes
$R/C$ nonzero. The injection
$(R/C)(-d)\hookrightarrow R/(f_1,f_2)$ shows that $C$ has height
two and positive multiplicity. Consequently,
\[
\deg(R/J)=d^2-\deg(R/C)\le d^2-1.
\]
We may therefore write
\[
\deg(R/J)=d^2-d+\alpha,
\qquad 0\le\alpha\le d-1.
\]
Lemma~\ref{upperboundinitialdegree} gives $1\le e\le d-1$.
Using the Bourbaki formula, we obtain
\[
1\le\Bour(J)
=e^2-ed+d-\alpha
=(e-1)(e+1-d)+1-\alpha
\le1-\alpha.
\]
Here the first inequality holds because $J$ has projective dimension
two and is therefore not perfect of height two. It follows that
$\alpha=0$, and hence $\deg(R/J)=d^2-d$. This proves the bound.
\end{proof}

\begin{Remark}
The value $d^2-d+1$ is not a global degree gap, since it occurs in the perfect case. Indeed, a height-two perfect ideal generated by three forms of degree $d$ with Hilbert--Burch syzygy degrees $1$ and $d-1$ has $\deg(R/J)=d^2-d+1$. Thus $d^2-d+1$ is a structural gap for ideals of homological dimension $2$, not for all strict almost complete intersections.
\end{Remark}

\begin{Theorem}\label{structural-gap-degree-general-desc}
Let $R=k[x_1,\ldots,x_n]$, with $n\ge3$, and let $J\subseteq R$ be a $d$-equigenerated strict almost complete intersection of height $2$ and homological dimension $2$. The following are equivalent:
\begin{enumerate}
\item[{\rm(a)}] $\deg(R/J)=d^2-d$;
\item[{\rm(b)}] $J$ has one of the following minimal graded free resolutions:
\[
0\to R(-2d-1)\to R(-d-1)\oplus R(-2d)^2
\to R(-d)^3\to J\to0,
\]
or $d=3$ and
\[
0\to R(-6)\to R(-5)^3\to R(-3)^3\to J\to0.
\]
\end{enumerate}
If either condition holds, then $\Bour(J)=1$.
\end{Theorem}

\begin{proof}
{\rm(a)} $\Rightarrow$ {\rm(b)}. By Lemma~\ref{upperboundinitialdegree},
$e=\indeg(\syz(J))\le d-1$. The Bourbaki degree formula gives
\[
\Bour(J)=e^2-ed+d=(e-1)(e+1-d)+1.
\]
Since $J$ has homological dimension $2$, it is not perfect, and hence
$\Bour(J)\ge1$. As $e-1\ge0$ and $e+1-d\le0$, the preceding equality
forces
\[
(e-1)(e+1-d)=0.
\]
Thus $e=1$ or $e=d-1$, and in either case $\Bour(J)=1$.

By Proposition~\ref{prop:not-equigen-Bour-consequences}, the Bourbaki
ideal $I_\nu$ associated to a minimal syzygy $\nu$ is generated by two
independent linear forms and has resolution
\[
0\to R(-2)\to R(-1)^2\to I_\nu\to0.
\]
Applying Theorem~\ref{Bour3Gens}{\rm(c)}, we obtain
\[
0\to R(-(d-e+2))
\to R(-e)\oplus R(-(d-e+1))^2
\to\syz(J)\to0,
\]
and hence
\[
0\to R(-(2d-e+2))
\to R(-(d+e))\oplus R(-(2d-e+1))^2
\to R(-d)^3\to J\to0.
\]
Since $e$ is the initial syzygy degree, minimality gives
$d-e+1\ge e$. If $e=d-1$, this implies $d\le3$; the case $d=2$ already
belongs to $e=1$, while $d=3$ gives $e=2$. Substituting $e=1$ or
$(d,e)=(3,2)$ yields the two resolutions in {\rm(b)}.

Conversely, the Hilbert series numerators corresponding to the two
resolutions in {\rm(b)} are, respectively,
\[
1-3t^d+t^{d+1}+2t^{2d}-t^{2d+1}
\quad\text{and}\quad
1-3t^3+3t^5-t^6.
\]
Taking the second derivative at $t=1$ and dividing by $2$ gives
$\deg(R/J)=d^2-d$ in both cases. The first resolution has $e=1$, while
the second has $(d,e)=(3,2)$; therefore the Bourbaki degree formula
gives $\Bour(J)=1$.
\end{proof}

	
\subsection{Ideals with exactly three syzygies}

When $m=3$ the minimal graded free resolution of $J$ from \eqref{res-prelim-bis} has the shape
\begin{equation}\label{resm=3}
0\to R(-\delta_1-\delta_2-\delta_3)\to R(-d-\delta_1)\oplus R(-d-\delta_2)\oplus R(-d-\delta_3) \to R(-d)^3\to J \to 0
\end{equation}
where $\delta_1\leq \delta_2\leq \delta_3\leq d$ with $\delta_i+\delta_j\geq d+1$ for every $1\leq i<j\leq 3$.

 Conversely, for each integer $d\geq 2$ and each triple $\delta_1\leq \delta_2\leq \delta_3\leq d$ with $\delta_i+\delta_j\geq d+1$ for every $1\leq i<j\leq 3,$ there is a codimension two ideal $J$ with minimal graded free resolution \eqref{resm=3}, except for the values $(\delta_1, \delta_2, \delta_3)=(d,d,d)$ resulting in the Koszul resolution for the complete intersection case.   In fact, if we consider the $4\times 3$ matrix
 
 $$\eta=\left[\begin{matrix}
x^{d-\delta_1}&0&0\\
y^{d-\delta_1}&x^{d-\delta_2}&0\\
0&y^{d-\delta_2}&x^{d-\delta_3}\\
x^{\delta_2+\delta_3-d}&z^{\delta_1+\delta_3-d}&y^{\delta_1+\delta_2-d} 
\end{matrix}\right]$$
we have that $\eta$ is a $(d,3,\delta_1,\delta_2,\delta_3,\delta_1+\delta_2-d)$-{\it level matrix} as in \cite[Definition 1.2]{BFRS}. Thus, by \cite[Theorem 1.3]{BFRS} the ideal generated by the maximal minors of $\eta$ fixing the last row is a homogeneous ideal of codimension 2 with minimal graded free resolution \eqref{resm=3}. In this case, we obtain the following formula for the Bourbaki degree. 

\begin{Proposition}\label{Bourm=3} Let $J$ be an ideal of codimension $2$ with minimal graded free resolution \eqref{resm=3}. Then: $${\rm Bour}(J)=(d-\delta_1-\delta_2)(d-\delta_1-\delta_3).$$
\end{Proposition}
 \begin{proof}
 From the  minimal graded free resolution  \eqref{resm=3} we get that the Hilbert series of $R/J$ is
$$H_{R/J}(t)=\frac{ S_{R/J}(t)}{(1-t)^n}$$
where $S_{R/J}(t)= 1-3t^d+t^{d+\delta_1}+ t^{d+\delta_2}+t^{d+\delta_3}-t^{\delta_1+\delta_2+\delta_3}$. So, 
$$\deg(R/J)=\frac{S''_{R/J}(1)}{2}=d(\delta_1+\delta_2+\delta_3)-(\delta_1\delta_2+\delta_1\delta_3+\delta_2\delta_3).$$
Consequently,
\begin{eqnarray}
{\rm Bour}(J)&=&d^2-\delta_1d+\delta_1^2-d(\delta_1+\delta_2+\delta_3)+(\delta_1\delta_2+\delta_1\delta_3+\delta_2\delta_3)\nonumber\\
&=& d^2-2\delta_1d+\delta_1^2-(\delta_2+\delta_3)(d-\delta_1)+\delta_2\delta_3\nonumber\\
&=&(d-\delta_1)^2-(\delta_2+\delta_3)(d-\delta_1)+\delta_2\delta_3\nonumber\\
&=&(d-\delta_1-\delta_2)(d-\delta_1-\delta_3)\nonumber
\end{eqnarray}
\end{proof}

\subsection{Ideals generated by quadrics}\label{subsec:ideals-quadrics} Let $J$ be a $2$-equigenerated ideal of codimension at least 2 and homological dimension 2 with minimal graded free resolution  
\begin{equation}
	0\to \bigoplus_{i=1}^{m-2}R(-2-\delta_{i+2}-\epsilon_i)\to \bigoplus_{i=1}^m R(-2-\delta_i)\to R(-2)^3\to J\to 0
\end{equation} 
By Remark~\ref{numberofsyzygies}, we have $3\leq m\leq 4$. In this section, for each $m$ in this range, we display all possible vectors $\underline{\delta}:=(\delta_1,\ldots,\delta_m)$ and   $\underline{\epsilon}:=(\epsilon_1,\ldots,\epsilon_{m-2})$ and their respective values for $\deg(R/J)$ and ${\rm Bour}(J)$.

\begin{Proposition}[$m=3$]
	Let $J$ be a $2$-equigenerated ideal of codimension at least two with minimal graded free resolution
	\begin{equation}
		0\to R(-\delta_1-\delta_2-\delta_3)\to \bigoplus_{i=1}^3 R(-2-\delta_i)\to R(-2)^3\to J\to 0
	\end{equation}
	Then, the list of all possible vectors $\underline{\delta}$ and $\underline{\epsilon}$ and their respective values $\deg(R/J)$ and ${\rm Bour}(J)$ is:
	\begin{table}[H]
		\begin{center}
			\begin{tabular}{|c|c|c|c|}
				\hline
				$\underline{\delta}$&$\underline{\epsilon}$&$\deg(R/J)$&{\rm Bour}(J)\\
				\hline
				$(1,2,2)$&$(1)$&$2$&$1$\\
				$(2,2,2)$&$(2)$&$0$&$4$\\
				\hline
					\end{tabular}
			\end{center}
		\end{table}
				\end{Proposition}
			\begin{proof} For $d=2$ and $m=3$ we have that the vectors $\underline{\delta}$ and  $\underline{\epsilon}$ that  satisfy the conditions \eqref{Dnondecresing}, \eqref{sumHS-c1}, \eqref{c2} and \eqref{c3} are exactly those listed in the table above. The respective ${\rm Bour}(J)$ can be calculated using Proposition~\ref{Bourm=3} and $\deg(R/J)$ can be calculated by the formula $\deg(R/J)=d^2-\delta_1d+\delta_1^2-{\rm Bour}(J).$
					\end{proof}

                    \begin{Proposition}[$m=4$]
				Let $J$ be a $2$-equigenerated ideal of codimension $2$ with minimal graded free resolution
			\begin{equation}\label{res-d2-m=4}
				0\to \bigoplus_{i=1}^2R(-2-\delta_{i+2}-\epsilon_i)\to \bigoplus_{i=1}^4 R(-2-\delta_i)\to R(-2)^3\to J\to 0
			\end{equation}
		Then the list of all possible vectors $\underline{\delta}$ and $\underline{\epsilon}$ and their respective values $\deg(R/J)$ and ${\rm Bour}(J)$ is:

			\begin{center}
				\begin{tabular}{|c|c|c|c|}
					\hline
			$\underline{\delta}$&$\underline{\epsilon}$&$\deg(R/J)$&{\rm Bour}(J)\\
					\hline
					$(2,2,2,2)$&$(1,1)$&$1$&$3$\\
					\hline
				\end{tabular}
			\end{center}
		
		\end{Proposition}
	\begin{proof}   For $d=2$ and $m=4$ we have that the vectors $\underline{\delta}$ and  $\underline{\epsilon}$ that  satisfy the conditions \eqref{Dnondecresing}, \eqref{sumHS-c1}, \eqref{c2} and \eqref{c3} are:

			\begin{center}
				\begin{tabular}{|c|c|}
					\hline
					$\underline{\delta}$&$\underline{\epsilon}$\\
					\hline
					$(2,2,2,\delta_4),\,2\leq \delta_4$&$(1,1)$\\
					\hline
				\end{tabular}
			\end{center}
		So, from  the minimal graded free resolution \eqref{res-d2-m=4} we have that the respective degrees  of $R/J$ are:
			\[\deg(R/J)=-\delta_4+3\]
			Since $\deg(R/J)>0$ we have $\delta_4=2.$
\end{proof}

We note in Table~\ref{table:gradient-triples-cubics} that all possible resolutions described above occur when $n=3$ for gradient ideals $J = J_F$ of cubics $\deg(F) = 3$, as described in \cite[Section 5]{Futata2023}:

\begin{table}[H]
\caption{Gradient ideals of plane projective cubics and resolutions}
\begin{tabular}{|l|l|l|l|l|l|l|}
\hline
Cubic                & Tjurina & Bourbaki & indeg & m & $\underline{\delta}$ & $\underline{\epsilon}$ \\ \hline
non-singular         & 0       & 4               & 2              & 3 & (2,2,2)                                            & (2)                                                  \\ \hline
nodal                & 1       & 3               & 2              & 4 & (2,2,2,2)                                          & (1,1)                                                \\ \hline
cuspidal             & 2       & 1               & 1              & 3 & (1,2,2)                                            & (1)                                                  \\ \hline
smooth conic + transverse line         & 2       & 1               & 1              & 3 & (1,2,2)                                            & (1)                                                  \\ \hline
smooth conic + tangent line & 3       & 0               & 1              & 2 & $(1,1)$                                               & free                                                 \\ \hline
3 generic lines      & 3       & 0               & 1              & 2 & $(1,1)$                                               & free    \\ \hline
\end{tabular}
\label{table:gradient-triples-cubics}
\end{table}
Note also that $\deg(R/J) = 3$ implies that $\syz(J)$ is free, from Theorem~\ref{structural-gap-degree-general}.

\subsection{Ideals generated by cubics}\label{subsec:ideals-cubics}
 
Let $J$ be a $3$-equigenerated ideal of codimension at least two with minimal graded free resolution  	\begin{equation}
	0\to \bigoplus_{i=1}^{m-2}R(-3-\delta_{i+2}-\epsilon_i)\to \bigoplus_{i=1}^m R(-3-\delta_i)\to R(-3)^3\to J\to 0
\end{equation} 
By the Remark~\ref{numberofsyzygies}, we have in this case  $3\leq m\leq 5.$
For each $m$ in this range we  will display all possible vectors $\underline{\delta}:=(\delta_1,\ldots,\delta_m)$ and   $\underline{\epsilon}:=(\epsilon_1,\ldots,\epsilon_{m-2})$ and their respective values $\deg(R/J)$ and ${\rm Bour}(J).$

\begin{Proposition}[$m=3$]
	Let $J$ be a $3$-equigenerated ideal of codimension at least two with minimal graded free resolution
	\begin{equation}
		0\to R(-\delta_1-\delta_2-\delta_3)\to \bigoplus_{i=1}^3 R(-3-\delta_i)\to R(-3)^3\to J\to 0
	\end{equation}
	Then, the list of all possible vectors $\underline{\delta}$ and $\underline{\epsilon}$ and their respective values $\deg(R/J)$ and ${\rm Bour}(J)$ is:
	\begin{table}[H]
		\begin{center}
			\begin{tabular}{|c|c|c|c|}
				\hline
				$\underline{\delta}$&$\underline{\epsilon}$&$\deg(R/J)$&{\rm Bour}(J)\\
				\hline
				$(1,3,3)$&$(1)$&$6$&$1$\\
				$(2,2,2)$&$(1)$&$6$&$1$\\
				$(2,2,3)$&$(1)$&$5$&$2$\\
				$(2,3,3)$&$(2)$&$3$&$4$\\
				$(3,3,3)$&$(3)$&$0$&$9$\\
				\hline
					\end{tabular}
			\end{center}
		\end{table}
				\end{Proposition}
			\begin{proof} For $d=3$ and $m=3$ we have that the vectors $\underline{\delta}$ and  $\underline{\epsilon}$ that  satisfy the conditions \eqref{Dnondecresing}, \eqref{sumHS-c1}, \eqref{c2} and \eqref{c3} are exactly those listed in the table above. The respective ${\rm Bour}(J)$ can be calculated using Proposition~\ref{Bourm=3} and $\deg(R/J)$ can be calculated by the formula $\deg(R/J)=d^2-\delta_1d+\delta_1^2-{\rm Bour}(J).$
					\end{proof}

		\begin{Proposition}[$m=4$]
				Let $J$ be a $3$-equigenerated ideal of codimension $2$ with minimal graded free resolution
			\begin{equation}\label{res-d3-m=4}
				0\to \bigoplus_{i=1}^2R(-3-\delta_{i+2}-\epsilon_i)\to \bigoplus_{i=1}^4 R(-3-\delta_i)\to R(-3)^3\to J\to 0
			\end{equation}
		Then the list of all possible vectors $\underline{\delta}$ and $\underline{\epsilon}$ and their respective values $\deg(R/J)$ and ${\rm Bour}(J)$ is:

			\begin{center}
				\begin{tabular}{|c|c|c|c|}
					\hline
			$\underline{\delta}$&$\underline{\epsilon}$&$\deg(R/J)$&{\rm Bour}(J)\\
					\hline
					$(2,3,3,3)$&$(1,1)$&$4$&$3$\\
					$(3,3,3,3)$&$(1,2)$&$2$&$7$\\
					$(3,3,3,4)$&$(2,1)$&$1$&$8$\\
					\hline
				\end{tabular}
			\end{center}
		
		\end{Proposition}
	\begin{proof}   For $d=3$ and $m=4$ we have that the vectors $\underline{\delta}$ and  $\underline{\epsilon}$ that  satisfy the conditions \eqref{Dnondecresing}, \eqref{sumHS-c1}, \eqref{c2} and \eqref{c3} are:

			\begin{center}
				\begin{tabular}{|c|c|}
					\hline
					$\underline{\delta}$&$\underline{\epsilon}$\\
					\hline
					$(2,3,3,\delta_4),\,3\leq \delta_4$&$(1,1)$\\
					$(3,3,3,\delta_4),\,3\leq \delta_4$&(1,2)\\
					$(3,3,3,\delta_4),\, 4\leq \delta_4$&(2,1)\\
					\hline
				\end{tabular}
			\end{center}
		So, from  the minimal graded free resolution \eqref{res-d3-m=4} we have that the respective degrees  of $R/J$ are:

		\begin{center}
			\begin{tabular}{|c|c|c|}
				\hline
				$\underline{\delta}$&$\underline{\epsilon}$&$\deg(R/J)$\\
				\hline
				$(2,3,3,\delta_4),\,3\leq \delta_4$&$(1,1)$&$-\delta_4+7$\\
				$(3,3,3,\delta_4),\,3\leq \delta_4$&(1,2)&$-2\delta_4+8$\\
				$(3,3,3,\delta_4),\, 4\leq \delta_4$&(2,1)&$-\delta_4+5$\\
				\hline
			\end{tabular}
		\end{center}
	
	Now we analyze the possibilities for $\delta_4$ in each case above.
	
	First, we consider the case $\underline{\delta}=(2,3,3,\delta_4),$ $3\leq \delta_4,$ and $\underline{\epsilon}=(1,1)$. If $4\leq \delta_4$  then a $5\times 4$ matrix with entries in $R$ satisfying the first two conditions in the definition
	of a $(3,4,\underline{\delta},\underline{\epsilon})$-level matrix
	must have the following format:
	$$\left[\begin{matrix}
	a_{1,1}&a_{1,2}&a_{1,3}&0\\	
	a_{2,1}&a_{2,2}&a_{2,3}&0\\
	a_{3,1}&a_{3,2}&a_{3,3}&0\\
	b_{1,1}&b_{1,2}&b_{1,3}&0\\
	b_{2,1}&b_{2,2}&b_{2,3}&b_{2,4}\\
	\end{matrix}\right]$$
But, for such a matrix $I_4(\eta)\subseteq(b_{2,4})$. Thus, there does not exist  $(3,4,\underline{\delta},\underline{\epsilon})$-level matrix with $\delta_4\geq 4$.  Thus, by \cite[Theorem 1.3]{BFRS}, there does not exist $3$-equigenerated ideal of codimension 2  with minimal graded free resolution \eqref{res-d3-m=4}  for which the shifts  afford the
vectors $(3,4,\underline{\delta},\underline{\epsilon})$ with  $\delta_4\geq 4$. Thus, in this case we have necessarily $\delta_4=3.$

Finally, for the other cases,  the fact that $\deg(R/J)\geq 1$ implies that $\delta_4=3$ in the second row of the table and $\delta_4=4$ in the third row of the table.

		\end{proof}
	
	\begin{Proposition}[$m=5$]
		Let $J$ be a $3$-equigenerated ideal of codimension $2$ with minimal graded free resolution
		\begin{equation}\label{resm=5}
			0\to \bigoplus_{i=1}^3R(-3-\delta_{i+2}-\epsilon_i)\to \bigoplus_{i=1}^5 R(-3-\delta_i)\to R(-3)^3\to J\to 0
		\end{equation}
	Then the list of all possible vectors $\underline{\delta}$ and $\underline{\epsilon}$ and their respective values $\deg(R/J)$ and ${\rm Bour}(J)$ is:
	\begin{table}[H]
		\renewcommand{\arraystretch}{1.15}
		\begin{center}
			\begin{tabular}{|c|c|c|c|}
				\hline
				$\underline{\delta}$&$\underline{\epsilon}$&$\deg(R/J)$&{\rm Bour}(J)\\
				\hline
			 $(3,3,3,3,3)$&$(1,1,1)$&$3$&$6$\\
			 \hline
			\end{tabular}
		\end{center}
	\end{table}
	
	\end{Proposition}

\begin{proof} For $d=3$ and $m=5$ we have that the vectors $\underline{\delta}$ and  $\underline{\epsilon}$ that  satisfy the conditions \eqref{Dnondecresing}, \eqref{sumHS-c1}, \eqref{c2} and \eqref{c3} are:
	$$\underline{\delta}=(3,3,3,\delta_4,\delta_5),\,\, 3\leq \delta_4\leq\delta_5\quad\mbox{and}\quad \underline{\epsilon}=(1,1,1)$$
	From the minimal graded free resolution \eqref{resm=5} we have that $\deg(R/J)=-\delta_4-\delta_5+9$. Since $\deg(R/J)\geq 1$ we conclude that $\delta_4+\delta_5\leq 8$.  Thus, $(\delta_4,\delta_5)$ is one of the following pairs: $(3,3),$ $(3,4),$ $(3,5),$ or $(4,4)$  
	
	For $(\delta_4,\delta_5)=(3,4)$ or $(3,5)$ a $6\times 5$ matrix satisfying the first two conditions in the definition
	of a $(3,5,\underline{\delta},\underline{\epsilon})$-level matrix
	must have the following format:
	$$\eta=\left[\begin{matrix}
		a_{1,1}&a_{1,2}&a_{1,3}&a_{1,4}&0\\
		a_{2,1}&a_{2,2}&a_{2,3}&a_{2,4}&0\\
		a_{3,1}&a_{3,2}&a_{3,3}&a_{3,4}&0\\
		b_{1,1}&b_{1,2}&b_{1,3}&b_{1,4}&0\\
		b_{2,1}&b_{2,2}&b_{2,3}&b_{2,4}&0\\
		b_{3,1}&b_{3,2}&b_{3,3}&b_{3,4}&b_{3,5}\\
	\end{matrix}\right]$$ 
But, for such a matrix $I_5(\eta)\subseteq(b_{3,5})$. Thus, there does not exist  $(3,m,\underline{\delta},\underline{\epsilon})$-level matrix with $(\delta_4,\delta_5)=(3,4)$ or $(3,5)$.  In particular, by \cite[Theorem 1.3]{BFRS}, there does not exist $3$-equigenerated ideal of codimension 2  with minimal graded free resolution \eqref{resm=5}  for which the shifts  afford the
vectors $(3,5,\underline{\delta},\underline{\epsilon})$ with  $(\delta_4,\delta_5)=(3,4)$ or $(3,5)$. 

On the other hand,	for $(\delta_4,\delta_5)=(4,4)$  a $6\times 5$ matrix satisfying the first two conditions in the definition
of a $(3,m,\underline{\delta},\underline{\epsilon})$-level matrix
must have the following format:

	$$\eta=\left[\begin{matrix}
	a_{1,1}&a_{1,2}&a_{1,3}&0&0\\
	a_{2,1}&a_{2,2}&a_{2,3}&0&0\\
	a_{3,1}&a_{3,2}&a_{3,3}&0&0\\
	b_{1,1}&b_{1,2}&b_{1,3}&0&0\\
	b_{2,1}&b_{2,2}&b_{2,3}&b_{2,4}&b_{2,5}\\
	b_{3,1}&b_{3,2}&b_{3,3}&b_{3,4}&b_{3,5}\\
\end{matrix}\right]$$
But, for such a matrix $I_5(\eta)\subseteq(b_{2,4}b_{3,5}-b_{2,5}b_{3,4})$. Thus, there does not exist  $(3,m,\underline{\delta},\underline{\epsilon})$-level matrix with $(\delta_4,\delta_5)=(4,4)$.  As before, by \cite[Theorem 1.3]{BFRS}, there does not exist $3$-equigenerated ideal of codimension 2  with minimal graded free resolution \eqref{resm=5}  for which the shifts  afford the
vectors $(3,5,\underline{\delta},\underline{\epsilon})$ with  $(\delta_4,\delta_5)=(4,4)$. 

Therefore, the only possible vectors for a $3$-equigenerated ideal of codimension 2  with minimal graded free resolution \eqref{resm=5} are $\underline{\delta}=(3,3,3,3,3)$ and $\underline{\epsilon}=(1,1,1)$. For these vectors, $\deg(R/J)=3$ and, from the Bourbaki degree formula, ${\rm Bour}(J)=6.$
	\end{proof}

As we will show in
Section~\ref{sec:quartic-curves}, after classifying all triples coming from reduced quartic plane projective curves, all of the possible values can be achieved by gradient ideals, just as in the case of cubics. Hence, for $d \leq 3$, there is no integrable gap (see Table~\ref{table:gradient-triples-quartics}).


\section{The Bourbaki degree of quartic plane curves}\label{sec:quartic-curves}

Let $R=k[x,y,z]$, where $k$ is algebraically closed of characteristic zero, and let $X=V(F)\subseteq \PP^2_k$ be a reduced singular quartic. We assume throughout this section that $X$ is not a cone, equivalently that the three partial derivatives $F_x,F_y,F_z$ are $k$-linearly independent. Thus the gradient ideal $J_F=(F_x,F_y,F_z)$ is minimally generated by three cubic forms and has height $2$. The excluded cone case is the reduced quartic consisting of four concurrent lines; in that case the gradient ideal is generated by only two cubic forms, so it does not belong to the three-generated setting of this paper.

Since $F$ is homogeneous of degree $4$ and $\operatorname{char}k=0$, Euler's formula gives
\[
4F=xF_x+yF_y+zF_z.
\]
Let $p\in \Sing(X)$, and choose an affine chart, say $z\neq 0$, with local equation
$g(u,v)=F(u,v,1)$. Then $g_u=F_x(u,v,1)$ and $g_v=F_y(u,v,1)$, while Euler's formula gives
\[
F_z(u,v,1)=4g-u g_u-v g_v.
\]
Hence, in the local ring $\mathcal O_{\PP^2,p}$, one has
\[
(J_F)_p=(g,g_u,g_v).
\]
Therefore, the local length of the singular scheme defined by $J_F$ at $p$ is the local Tjurina number $\tau_p(X)$. Consequently,
\[
\deg(R/J_F)=\tau(X):=\sum_{p\in\Sing(X)}\tau_p(X).
\]

If $e:=\indeg(\syz(J_F))$, then the Bourbaki formula for ideals minimally generated by three cubic forms gives
\begin{equation}\label{eq:quartic-bour-formula}
\Bour(X):=\Bour(J_F)=9-3e+e^2-\tau(X).
\end{equation}

For $C\subseteq \PP^2$ a reduced plane curve of degree $d\geq 3$ and a singular point $p\in C$, let $\mu_p(C)$ denote the Milnor number and let $r_p(C)$ denote the number of analytic branches of $C$ at $p$. The $\delta$-invariant is given by
\[
\delta_p(C)=\frac{\mu_p(C)+r_p(C)-1}{2}.
\]
If $C$ is irreducible, then the genus formula (see, for example, \cite[Remark 2.1.11]{namba1984geometry}) gives
\begin{equation}\label{eq:genus-formula}
g(C)=\frac{(d-1)(d-2)}{2}-\sum_{p\in\Sing(C)}\delta_p(C)
\end{equation}
where $g(C)$ is the genus of the normalization of $C$, or the \emph{geometric genus}.

We now specialize to quartics. If $X\subseteq \PP^2$ is an irreducible quartic, then \eqref{eq:genus-formula} gives
\[
g(X)=3-\sum_{p\in\Sing(X)}\delta_p(X).
\]
Since $g(X)\geq 0$ and $\delta_p(X)\geq 1$ for every singular point, one has
\[
\sum_{p\in\Sing(X)}\delta_p(X)\leq 3.
\]
In particular, an irreducible singular quartic has at most three singular points. Moreover, among the simple singularities, only those with $\delta\leq 3$ can occur on an irreducible quartic. Thus, the possible local types are
\[
A_1,\ldots,A_6,\qquad D_4,\qquad D_5,\qquad E_6.
\]
The actual classification of the possible configurations, together with the relevant analytic restrictions and normal forms, is due to Wall; see \cite{Wall-quartics}. We shall use only the induced list of configurations and their numerical invariants. For related treatments of quartic strata and Jacobian ideals, see also \cite{Farrahy-NN,NN-Simis-Aluffi}.

For reducible quartics, the genus formula is less useful for our purpose. Instead, one proceeds geometrically from the possible decompositions:
\[
(3+1),\qquad (2+2),\qquad (2+1+1),\qquad (1+1+1+1),
\]
that is, a cubic plus a line, two conics, a conic plus two lines, or four lines. These cases give the remaining configurations, including the singularities $A_7$, $D_6$, and $E_7$, except for the cone case of four concurrent lines, which is excluded from our three-generated setting. Again, the corresponding normal forms can be found in \cite{Farrahy-NN}. In the present paper, we use only the resulting singularity configurations and their numerical invariants.

\begin{table}[H]
\caption{Numerical invariants of singularities occurring on reduced quartics.}
\centering
\renewcommand{\arraystretch}{1.15}
\begin{tabular}{|c|c|c|c|}
\hline
Type & Description & $\mu_p=\tau_p$ & $\delta_p$ \\
\hline
$A_1$ & node & $1$ & $1$ \\
$A_2$ & simple cusp & $2$ & $1$ \\
$A_3$ & tacnode & $3$ & $2$ \\
$A_4$ & ramphoid cusp & $4$ & $2$ \\
$A_5$ & oscnode & $5$ & $3$ \\
$A_6$ & higher cusp & $6$ & $3$ \\
$A_7$ & higher tacnode & $7$ & $4$ \\
$D_4$ & ordinary triple point & $4$ & $3$ \\
$D_5$ & tacnode cusp & $5$ & $3$ \\
$D_6$ & higher triple point & $6$ & $4$ \\
$E_6$ & cusp of multiplicity $3$ & $6$ & $3$ \\
$E_7$ & cusp with a tangent smooth branch & $7$ & $4$ \\
\hline
\end{tabular}
\label{tab:quartic-local-invariants}
\end{table}

All singularities listed in Table~\ref{tab:quartic-local-invariants} are weighted homogeneous. Hence, by Saito's criterion, the local Tjurina and Milnor numbers coincide. Thus, if the singularity configuration of $X$ is written as
\[
\Sigma_X=\sum_{i=1}^{7}\lambda_iA_i+\eta_4D_4+\eta_5D_5+\eta_6D_6+\theta_6E_6+\theta_7E_7,
\]
then
\begin{equation}\label{eq:quartic-tjurina-configuration}
\tau(X)=\sum_{i=1}^{7}i\lambda_i+4\eta_4+5\eta_5+6\eta_6+6\theta_6+7\theta_7.
\end{equation}

We shall also use the classification of three-generated cubic ideals from Section~\ref{subsec:ideals-cubics}. In the present setting, it implies the following restrictions on the initial degree $e=e(X)$:
\[
\begin{array}{|c|c|}
\hline
\tau(X) & e\\
\hline
\tau(X)\geq4 & e\leq2\\
\tau(X)=3 & e=2\text{ or }3\\
\tau(X)=2 & e=3\\
\tau(X)=1 & e=3\\
\hline
\end{array}
\]
Indeed, these are the restrictions supplied by the classification of possible pairs $(e,\deg(R/J_F))$ for three cubic generators, with $\deg(R/J_F)=\tau(X)$.

The following theorem characterizes the Bourbaki degree of a reduced singular quartic plane curve in terms of the configuration of its singularities.

\begin{Theorem}\label{Bour-Quartic}
Let $X=V(F)\subseteq \PP^2$ be a reduced singular quartic curve which is not a cone, and write its singularity configuration as
$$
\Sigma_X=\sum_{i=1}^{7}\lambda_iA_i+\eta_4D_4+\eta_5D_5+\eta_6D_6+\theta_6E_6+\theta_7E_7,
$$
where $\lambda_i,\eta_j,\theta_k\in\NN$. Then the Bourbaki degree
$\Bour(X)$ is characterized by Table~\ref{quartic-bour-table}, where
all coefficients not displayed in a given row are understood to be zero.
\end{Theorem}

{\small 
\begin{table}[htp]
\caption{Characterization of $\Bour(X)$.}
\centering
\renewcommand{\arraystretch}{1.15}
\begin{tabular}{|c|c|c|}
\hline
$\Bour(X)$ &  coefficients  of $\Sigma_X$ & Singularity configuration \\ 
\hline
\multirow{6}{*}{$\Bour(X)=0$}
& $(\lambda_2,\lambda_5)=(1,1)$ & $A_2+A_5$ \\
& $\lambda_7=1$ & $A_7$ \\
& $(\lambda_1,\lambda_3)=(1,2)$ & $A_1+2A_3$ \\
& $(\lambda_1,\eta_6)=(1,1)$ & $A_1+D_6$ \\
& $(\lambda_1,\eta_4)=(3,1)$ & $3A_1+D_4$ \\
& $\theta_7=1$ & $E_7$ \\
\hline
\multirow{12}{*}{$\Bour(X)=1$}
& $\lambda_2=3$ & $3A_2$ \\
& $(\lambda_2,\lambda_4)=(1,1)$ & $A_2+A_4$ \\
& $\lambda_6=1$ & $A_6$ \\
& $\theta_6=1$ & $E_6$ \\
& $(\lambda_1,\lambda_5)=(1,1)$ & $A_1+A_5$ \\
& $(\lambda_1,\lambda_2,\lambda_3)=(1,1,1)$ & $A_1+A_2+A_3$ \\
& $(\lambda_1,\eta_5)=(1,1)$ & $A_1+D_5$ \\
& $\eta_6=1$ & $D_6$ \\
& $\lambda_3=2$ & $2A_3$ \\
& $(\lambda_1,\lambda_3)=(3,1)$ & $3A_1+A_3$ \\
& $(\lambda_1,\eta_4)=(2,1)$ & $2A_1+D_4$ \\
& $\lambda_1=6$ & $6A_1$ \\
\hline
\multirow{9}{*}{$\Bour(X)=2$}
& $(\lambda_1,\lambda_2)=(1,2)$ & $A_1+2A_2$ \\
& $(\lambda_2,\lambda_3)=(1,1)$ & $A_2+A_3$ \\
& $(\lambda_1,\lambda_4)=(1,1)$ & $A_1+A_4$ \\
& $\lambda_5=1$ & $A_5$ \\
& $\eta_5=1$ & $D_5$ \\
& $(\lambda_1,\lambda_3)=(2,1)$ & $2A_1+A_3$ \\
& $(\lambda_1,\eta_4)=(1,1)$ & $A_1+D_4$ \\
& $(\lambda_1,\lambda_2)=(3,1)$ & $3A_1+A_2$ \\
& $\lambda_1=5$ & $5A_1$ \\
\hline
\multirow{6}{*}{$\Bour(X)=3$}
& $(\lambda_1,\lambda_2)=(2,1)$ & $2A_1+A_2$ \\
& $(\lambda_1,\lambda_3)=(1,1)$ & $A_1+A_3$ \\
& $\eta_4=1$ & $D_4$ \\
& $\lambda_2=2$ & $2A_2$ \\
& $\lambda_4=1$ & $A_4$ \\
& $\lambda_1=4$ & $4A_1$ \\
\hline
\multirow{2}{*}{$\Bour(X)=4$ or $\Bour(X)=6$}
& $\lambda_1=3$ & $3A_1$ \\
& $\lambda_3=1$ & $A_3$ \\
\hline
$\Bour(X)=6$
& $(\lambda_1,\lambda_2)=(1,1)$ & $A_1+A_2$ \\
\hline
\multirow{2}{*}{$\Bour(X)=7$}
& $\lambda_1=2$ & $2A_1$ \\
& $\lambda_2=1$ & $A_2$ \\
\hline
$\Bour(X)=8$ & $\lambda_1=1$ & $A_1$ \\
\hline
\multicolumn{3}{|c|}{$\Bour(X)=5$ never occurs for a reduced singular quartic curve.}\\
\hline
\end{tabular}
\label{quartic-bour-table}
\end{table}}
\begin{proof}
Since $J_F$ is generated by three cubic forms and $\deg(R/J_F)=\tau(X)$, formula~\eqref{eq:quartic-bour-formula} gives
\[
\Bour(X)=9-3e+e^2-\tau(X).
\]
Moreover, by \eqref{eq:quartic-tjurina-configuration}, the Tjurina number is read directly from the singularity configuration.

We distinguish cases according to $\tau(X)$. If $\tau(X)\geq 4$, then by the cubic classification recalled above one has $e\leq 2$. Since the expression $9-3e+e^2$ has the same value $7$ for $e=1$ and $e=2$, we get
\[
\Bour(X)=7-\tau(X).
\]
Thus the configurations with $\tau(X)=7,6,5,4$ give $\Bour(X)=0,1,2,3$, respectively. These are precisely the first four blocks in Table~\ref{quartic-bour-table}.

If $\tau(X)=3$, then the possible configurations are
\[
3A_1,\qquad A_1+A_2,\qquad A_3.
\]
By the cubic classification recalled above, one has $e=2$ or $e=3$.
For the configurations $3A_1$ and $A_3$, both possibilities can occur,
and hence \eqref{eq:quartic-bour-formula} gives
\[
\Bour(X)=
\begin{cases}
4, & \text{if } e=2,\\
6, & \text{if } e=3.
\end{cases}
\]

For the configuration $A_1+A_2$, we claim that necessarily $e=3$. The syzygy-defect formula from \cite[Theorem 1]{dimca2013syzygies} applied for $\tau(X) = 3$ gives
$$
\dim_k \syz(J_F)_2 = 3 - \dim_k(R_1/(I_Z)_1) = \dim_k (I_Z)_1,
$$
where $I_Z = J_F^{\text{sat}}$, and thus $e(X) = 2$ if and only if the singular scheme $Z = \Sing(X)$ is contained scheme-theoretically on a line. 

To see this is impossible for $A_1 + A_2$ singularities, we first note that these must be irreducible. Otherwise, a reduced quartic retaining an ordinary cusp would consist of a cuspidal cubic and a line avoiding the cusp, whose intersections cannot produce just one node. Now, for an irreducible quartic of singularity $A_1 + A_2$, if a line contained the singular scheme, this line would pass through the node and be tangent to the cusp, so its intersection multiplicities would be at least $2 + 3 = 5$, contradicting Bézout. Thus $e(X) = 3$ always for $A_1 + A_2$
consequently $\Bour(X)=6$ for every quartic with singularity configuration $A_1+A_2$.

If $\tau(X)=2$, then the only possible configurations are
\[
2A_1,\qquad A_2.
\]
Here $e=3$, and therefore \eqref{eq:quartic-bour-formula} gives $\Bour(X)=7$.

Finally, if $\tau(X)=1$, then $X$ has a unique node, so the configuration is $A_1$. Again $e=3$, and hence $\Bour(X)=8$.

The value $\Bour(X)=5$ cannot occur. Indeed, since
$\Bour(X)\leq e^2$ and $e\leq3$, this value would force $e=3$.
Formula~\eqref{eq:quartic-bour-formula} would then give
$\tau(X)=4$, contradicting the cubic classification above, which gives
$e\leq2$ whenever $\tau(X)\geq4$. This completes the proof.
\end{proof}
\begin{Remark}\label{rmk:ambiguity-sing-type}
The theorem shows that, except for the configurations \(
3A_1\) and \(A_3\), the Bourbaki degree is determined by the singularity configuration. Another case where the singularity configuration does not necessarily determine the initial degree $e(X)$ is for type $E_6$, where $e(X) \leq 2$, but both possibilities correspond to the same value for the Bourbaki degree. The following table attests to the existence of such cases:
\begin{table}[H]
\centering
\renewcommand{\arraystretch}{1.15}
\begin{tabular}{|l|c|c|}
\hline
Quartic $f$ & Singularity type & $e(X)$ \\ \hline
$z(x^3+y^3+z^3)$
& $3A_1$ & $2$ \\ \hline

$x^2y^2+y^2z^2+z^2x^2$
& $3A_1$ & $3$ \\ \hline

$x^4+y^4+y^2z^2$
& $A_3$ & $2$ \\ \hline

$x^4+y^4+y^2z^2+x^2yz$
& $A_3$ & $3$ \\ \hline

$x^4+y^3z$
& $E_6$ & $1$ \\ \hline

$x^4+y^3z+x^2y^2$
& $E_6$ & $2$ \\ \hline
\end{tabular}
\end{table}
For $3A_1$ and $A_3$, the two possibilities occur
\[
\Bour(X)=4 \iff e(X)=2,
\qquad
\Bour(X)=6 \iff e(X)=3,
\]
and this is the only ambiguity in $\Bour(X)$ for a fixed singularity configuration for reduced projective quartic plane curves.
\end{Remark}

\begin{Remark}
The missing value $\Bour(X)=5$ is the quartic manifestation of the
low-degree gap for three-generated cubic ideals. It is not caused by the
absence of singularity configurations with $\tau=4$; rather, it follows
from the impossibility of the boundary pair $(\tau,e)=(4,3)$.
\end{Remark}

Comparing with Subsection~\ref{subsec:ideals-cubics}, we obtain the
following table, showing that every possible resolution occurs for a
gradient ideal of a quartic. For completeness, we also include the
nonsingular case.

\begin{table}[htp]
\centering
\small
\caption{Gradient ideals of plane projective quartics and resolutions}\label{table:gradient-triples-quartics}
\resizebox{\textwidth}{!}{%
\begin{tabular}{|l|l|l|l|l|l|l|}
\hline
Quartics          & Tjurina & Bourbaki & indeg & m & $\underline{\delta}$ & $\underline{\epsilon}$ \\ \hline
non-singular         & 0       & 9               & 3              & 3 & (3,3,3)                                            & (3)\\ \hline
$A_1$ (nodal)              & 1       & 8               & 3              & 4 & (3,3,3,4)                                          & (2,1)\\ \hline
$2A_1, A_2$     & 2       & 7               & 3              & 4 & (3,3,3,3)                                            & (1,2)\\ \hline
$3A_1, 3A_1^{(3+1)}, A_3^{\text{irr}}$         & 3       & 4               & 2              & 3 & (2,3,3)                                            & (2)\\ \hline
$3A_1^{\text{irr}}, A_1 + A_2, A_3$ & 3       & 6               & 3              & 5 & $(3,3,3,3,3)$  & (1,1,1)                                                 \\ \hline
$2A_1 + A_2, A_1+A_3, D_4, 2A_2, A_4, 4A_1$      & 4       & 3               & 2              & 4 & $(2,3,3,3)$ & (1,1)    \\ \hline
$A_1+2A_2, A_2+A_3, A_1+A_4, A_5, D_5, \ldots$      & 5       & 2               & 2              & 3 & $(2,2,3)$ & (1)    \\ \hline
$E_6$      & 6       & 1               & 1              & 3 & $(1,3,3)$ & (1)    \\ \hline
$A_6, 3A_2, E_6, 6A_1, A_1 + A_5, A_2 + A_4, \ldots$      & 6       & 1               & 2              & 3 & $(2,2,2)$ & (1)    \\ \hline
$E_7, A_7, A_2+A_5, A_1 + 2A_3, A_1+D_6,3A_1+D_4$      & 7       & 0               & 1              & 2 & $(1,2)$ & ---\\ \hline
\end{tabular}%
}
\end{table}


\begin{thebibliography}{99}

\bibitem{bruns1998cohen}
W.~Bruns and H.~J.~Herzog,
\emph{Cohen--Macaulay Rings},
revised ed., Cambridge Studies in Advanced Mathematics, vol.~39,
Cambridge University Press, Cambridge, 1998.


\bibitem{BFRS}
R. Burity, T. Fiel, Z. Ramos and A. Simis,
\emph{The structure of almost Cohen-Macaulay 3-generated ideals of codimension 2 in terms of matrix theory},
J. Algebra, v. 713, 2027, 491-514.

\bibitem{CTC-Plessis}
A. A. du Plessis and C. T. C. Wall,
\emph{Application of the theory of the discriminant to highly singular plane curves},
Math. Proc. Cambridge Philos. Soc.
126 (1999), 259--266.

\bibitem{dimca2013syzygies}
A.~Dimca,
\emph{Syzygies of Jacobian ideals and defects of linear systems},
Bull. Math. Soc. Sci. Math. Roumanie
56(104) (2013), no.~2, 191--203.

\bibitem{dimca2017exponents}
A.~Dimca and G.~Sticlaru,
\emph{On the exponents of free and nearly free projective plane curves},
Rev. Mat. Complut. 30 (2017), no.~2, 259--268.

\bibitem{Farrahy-NN}
A. B. Farrahy and A. Nasrollah Nejad,
\emph{Hypersurfaces with linear type singular loci},
J. Algebra Appl. 19 (2020), no. 9, 2050169.

\bibitem{Futata2023}
D. I. Futata,
\emph{Exploring N-freeness and numerical invariants of logarithmic tangent sheaves in reduced hypersurfaces: a stratified approach with algorithmic implementation},
Ph.D. thesis,
Universidade Estadual de Campinas,
2023.

\bibitem{MFA}
M. Jardim, F. Monteiro and A. Nasrollah Nejad,
\emph{The Bourbaki degree of the syzygy module of $2\times 4$ matrices},
\href{https://arxiv.org/abs/2604.04252}{arXiv:2604.04252}, 2026.

\bibitem{MAA}
M. Jardim, A. Nasrollah Nejad and A. Simis,
\emph{The Bourbaki degree of a plane projective curve},
Trans. Amer. Math. Soc. 377 (2024), no. 11, 7633--7655.

\bibitem{namba1984geometry}
M.~Namba,
\emph{Geometry of Projective Algebraic Curves},
Monographs and Textbooks in Pure and Applied Mathematics, vol.~88,
Marcel Dekker, New York, 1984.

\bibitem{NN-Simis-Aluffi}
A. Nasrollah Nejad and A. Simis,
\emph{The Aluffi algebra},
J. Singul. 3 (2011), 20--47.

\bibitem{Wall-quartics}
C. T. C. Wall,
\emph{Geometry of quartic curves},
Math. Proc. Cambridge Philos. Soc. 117 (1995), no. 3, 415--423.

\end{thebibliography}
\end{document}